\documentclass[a4paper,11pt]{article}

\usepackage[utf8]{inputenc}
\usepackage[T1]{fontenc}
\usepackage{lmodern}
\usepackage[ngerman,english]{babel}
\usepackage{amsmath, amsfonts, amssymb, amsthm, mathtools, dsfont, tikz, booktabs, mathrsfs, subcaption}
\usepackage{hyperref}
\usepackage[noabbrev]{cleveref}

\usetikzlibrary{arrows.meta,calc}

\usepackage[left=2.5cm,right=2.5cm,top=2.5cm,bottom=2.5cm]{geometry} 

\allowdisplaybreaks[2]

\makeatletter
\let\@fnsymbol\@alph
\makeatother

\definecolor{OSNAred}{HTML}{AD1034}
\definecolor{OSNAgold}{HTML}{FBB900}
\definecolor{OSNAgray}{HTML}{CFD0D1}
\definecolor{RUBblue}{HTML}{003560}  
\definecolor{RUBgreen}{HTML}{94C11C} 
\definecolor{RUBgray}{HTML}{E6E4E4}  

\theoremstyle{plain}

\newtheorem{thm}{Theorem}[section]

\newtheorem{lemma}[thm]{Lemma}

\newtheorem{proposition}[thm]{Proposition}

\newtheorem{rem}[thm]{Remark}

\crefname{thm}{theorem}{theorems}

\usepackage{enumitem}

\setlist[enumerate]{label=(\alph*)}

\newlist{maincaseswrap}{enumerate}{1}
\setlist[maincaseswrap]{%
	label=Case~\arabic*:,%
	ref=\arabic*,%
	font=\bfseries\boldmath,%
	wide,%
	labelindent=0pt,%
	topsep=1\baselineskip,%
	itemsep=1\baselineskip%
}
\crefname{maincaseswrapi}{case}{cases}
\newenvironment{maincases}{%
    \let\standarditem\item%
    \begin{maincaseswrap}%
    \renewcommand{\item}[1][]{\standarditem{\bfseries\boldmath{}##1\ifx##1\relax\else.\fi}\hspace{0.5em}}%
}{%
    \end{maincaseswrap}%
}

\newlist{subcaseswrap}{enumerate}{10}
\setlist[subcaseswrap]{%
	label=Subcase~\arabic{maincaseswrapi}.\arabic*:,%
	ref=\arabic{maincaseswrapi}.\arabic*,%
	font=\sffamily,%
	wide,%
	labelindent=0pt,%
	topsep=0.5\baselineskip,%
	itemsep=0.5\baselineskip%
}
\crefname{subcaseswrapi}{subcase}{subcases}
\newenvironment{subcases}{%
    \begin{subcaseswrap}%
    \renewcommand{\item}[1][]{\standarditem{\sffamily{}##1\ifx##1\relax\else.\fi}\hspace{0.25em}}%
}{%
    \end{subcaseswrap}%
}

\newlist{boundingcaseswrap}{enumerate}{1}
\setlist[boundingcaseswrap]{%
	label=Bounding,%
	ref=\arabic*,%
	font=\bfseries\boldmath,%
	wide,%
	labelindent=0pt,%
	topsep=1\baselineskip,%
	itemsep=1\baselineskip%
}
\crefname{boundingcaseswrapi}{boundingcase}{boundingcases}
\newenvironment{boundingcases}{%
    \let\standarditem\item%
    \begin{boundingcaseswrap}%
    \renewcommand{\item}[1][]{\standarditem{\bfseries\boldmath{}##1\ifx##1\relax\else.\fi}\hspace{0.5em}}%
}{%
    \end{boundingcaseswrap}%
}

\newlist{boundingsubcaseswrap}{enumerate}{1}
\setlist[boundingsubcaseswrap]{%
	label=Case~\arabic*:,%
	ref=\arabic*,%
	font=\sffamily,%
	wide,%
	labelindent=0pt,%
	topsep=0.5\baselineskip,%
	itemsep=0.5\baselineskip%
}
\crefname{boundingsubcaseswrapi}{boundingsubcase}{boundingsubcases}
\newenvironment{boundingsubcases}{%
    \begin{boundingsubcaseswrap}%
    \renewcommand{\item}[1][\relax]{\standarditem{\sffamily{}##1\ifx##1\relax\else.\fi}\hspace{0.25em}}%
}{%
    \end{boundingsubcaseswrap}%
}

\newlist{conditions}{enumerate}{2}
\setlist[conditions]{label=(\alph*)}
\crefname{conditionsi}{condition}{conditions}

\newlist{enumcases}{enumerate}{2}
\setlist[enumcases,1]{%
	label=Case~\arabic*:,%
	ref=\arabic*,%
	font=\bfseries\boldmath,%
	wide,%
	labelindent=0pt,%
	topsep=1\baselineskip,%
	itemsep=1\baselineskip%
}
\setlist[enumcases,2]{%
	label=Subcase~\arabic{enumcasesi}.\arabic*:,%
	ref=\arabic{enumcasesi}.\arabic*,%
	font=\sffamily,%
	wide,%
	labelindent=0pt,%
	topsep=0.5\baselineskip,%
	itemsep=0.5\baselineskip%
}
\crefname{enumcasesi}{case}{cases}
\crefname{enumcasesii}{subcase}{subcases}

\def\arr{\overrightarrow}
\def\indicator{\mathbf{1}}
\def\ee{\mathbf{e}}
\def\conv{\operatorname{conv}}
\def\D{\operatorname{D}}

\def\SEP{P}

\def\PP{\mathbb{P}}
\def\EE{\mathbb{E}}
    \def\E{\mathbb{E}}

    \def\var{\mathbb{V}}
    \def\VV{\mathbb{V}}
    \def\V{\mathbb{V}}

\def\ZZ{\mathbb{Z}}
\def\RR{\mathbb{R}}
\def\NN{\mathbb{N}}

\title{Central limit theorems for high dimensional lattice polytopes: symmetric edge polytopes} 
\author{
Torben Donzelmann\thanks{Universit\"at Osnabr\"uck, Germany. \url{torben.donzelmann@uos.de}}
\and
Martina Juhnke\thanks{Universit\"at Osnabr\"uck, Germany. \url{martina.juhnke@uos.de}}
\and
Benedikt Redno\ss\thanks{Ruhr University Bochum, Germany. \url{benedikt.rednoss@rub.de}}
\and
Christoph Th\"ale\thanks{Ruhr University Bochum, Germany. \url{christoph.thaele@rub.de}}
}
\date{}

\usepackage{todonotes}
\reversemarginpar
\begin{document}

\maketitle

\begin{abstract}
\noindent  We investigate symmetric edge polytopes generated by Erd\H{o}s--R\'enyi random graphs in a high-dimensional regime. These objects provide a natural and largely unexplored model of random lattice polytopes, in which geometric properties are governed by graph-theoretic structure. Focusing on the number of polytope edges and on the number of edges in unimodular triangulations, we derive precise asymptotics for expectations and variances and establish central limit theorems with explicit rates of convergence. Our analysis combines a detailed combinatorial-geometric study of the graph configurations determining the facial structure with the discrete Malliavin--Stein method for normal approximation. In particular, we identify a distinguished parameter value at which the leading variance term cancels, producing an atypical fluctuation regime. To the best of our knowledge, the results obtained here constitute the first distributional limit theorems for random lattice polytopes.\\[0.5cm]
\end{abstract}


\section{Introduction and main results}

\subsection{Motivation and background}

 Random polytopes have a long history in convex geometry and geometric probability, with the \emph{classical model} occupying a central place. In this model, one considers the convex hull of points drawn at random from a continuous distribution in~\(\RR^d\). Canonical examples include the uniform distribution on a convex body or on its boundary, as well as the beta-type or the standard Gaussian distribution. Fundamental questions concerning the volume of these random convex hulls, the number of their faces, and various shape characteristics have been studied extensively in the asymptotic regime where the number of generating points tends to infinity while the ambient dimension~\(d\) is kept fixed. In contrast, systematic investigations of random polytopes in high dimensions have only begun to emerge in recent years, driven by advances in asymptotic geometric analysis, random matrix theory and stochastic geometry. We refer the reader to the survey articles \cite{BaranySurvey,HugSurvey,ReitznerSurvey,SchneiderSurvey} and to the dedicated chapters on random polytopes in the monographs \cite{IsotropicConvexBodies,KST26,sw08} for comprehensive overviews of this area.

In contrast to the continuous setting, little attention has been devoted to \emph{random lattice polytopes}, which are obtained by selecting points in the integer lattice~\(\mathbb{Z}^d\) according to a probabilistic rule. The most common approach in the existing literature starts from a deterministic convex body and considers the convex hull of the lattice points it contains after a suitable randomization of the body or of the underlying lattice structure, see for instance \cite{BaranyMatousek,NgocReitzner}. Prior work has largely focused on expectations of geometric and combinatorial quantities and has been restricted to fixed ambient dimensions. To the best of our knowledge, distributional limit theorems for random lattice polytopes have not yet been established, in particular in the high-dimensional regime.

We consider two closely related classes of polytopes arising from graph-based constructions: the \emph{symmetric edge polytopes} and the \emph{cosmological polytopes}, originally introduced in \cite{OhsugiHibi} and \cite{CosmologicalIntorudction}, respectively. Both constructions give rise to lattice polytopes which encode fundamental graph-theoretic information in geometric form. The present paper is devoted to a model of random symmetric edge polytopes, while a second, separate paper will address the corresponding questions for random cosmological polytopes.

We now turn to symmetric edge polytopes, which are the central objects of the present paper, and recall their definition while fixing notation. Let \(G = (V, E)\) be a finite graph with set of \emph{nodes}
\[
V = [n] \coloneqq \{1, 2, \dots, n\}
\]
and set of \emph{arcs} \(E\subseteq \binom{V}{2}\coloneqq \{e~:~e\subseteq V, \; |e|=2\}\).
To avoid confusion, we will use the terms \emph{vertex} and \emph{edge} exclusively for $0$- and $1$-dimensional faces of a polytope, respectively. In the following, we use \((\ee_i)_{i\in V}\) to denote the standard basis of \(\RR^{V}\).

The \emph{symmetric edge polytope} \(\SEP_G\) associated to \(G\) is constructed as the convex hull of the points \(\pm (\ee_i - \ee_j)\) for each arc \(\{i,j\} \in E\), that is,
\[
\SEP_G \coloneqq \conv\bigl\{ \, \pm(\ee_i - \ee_j)~ :~ \{i,j\} \in E \, \bigr\} \subset \RR^{V},
\]
see \Cref{fig:SEPex} for examples. 
One may view this construction as embedding the adjacency structure of $G$ into \(\RR^{V}\). In this way, combinatorial features of the graph, most notably the presence or absence of short cycles, such as $3$- and $4$-cycles, are reflected in the facial structure of the polytope, in particular in its edges.
This connection has been exploited using methods from combinatorial commutative algebra and discrete geometry to derive detailed information on lattice point enumeration (for instance via the $h^\ast$-vector) and to construct unimodular triangulations of symmetric edge polytopes, see, for example, \cite{CodenottiVenturello-SEP,gammavectorsymmetricedgepoyltope,Arithmetic_apects_SEP,KT2023,MHN2011,OhsugiHibi,OH2014}. 

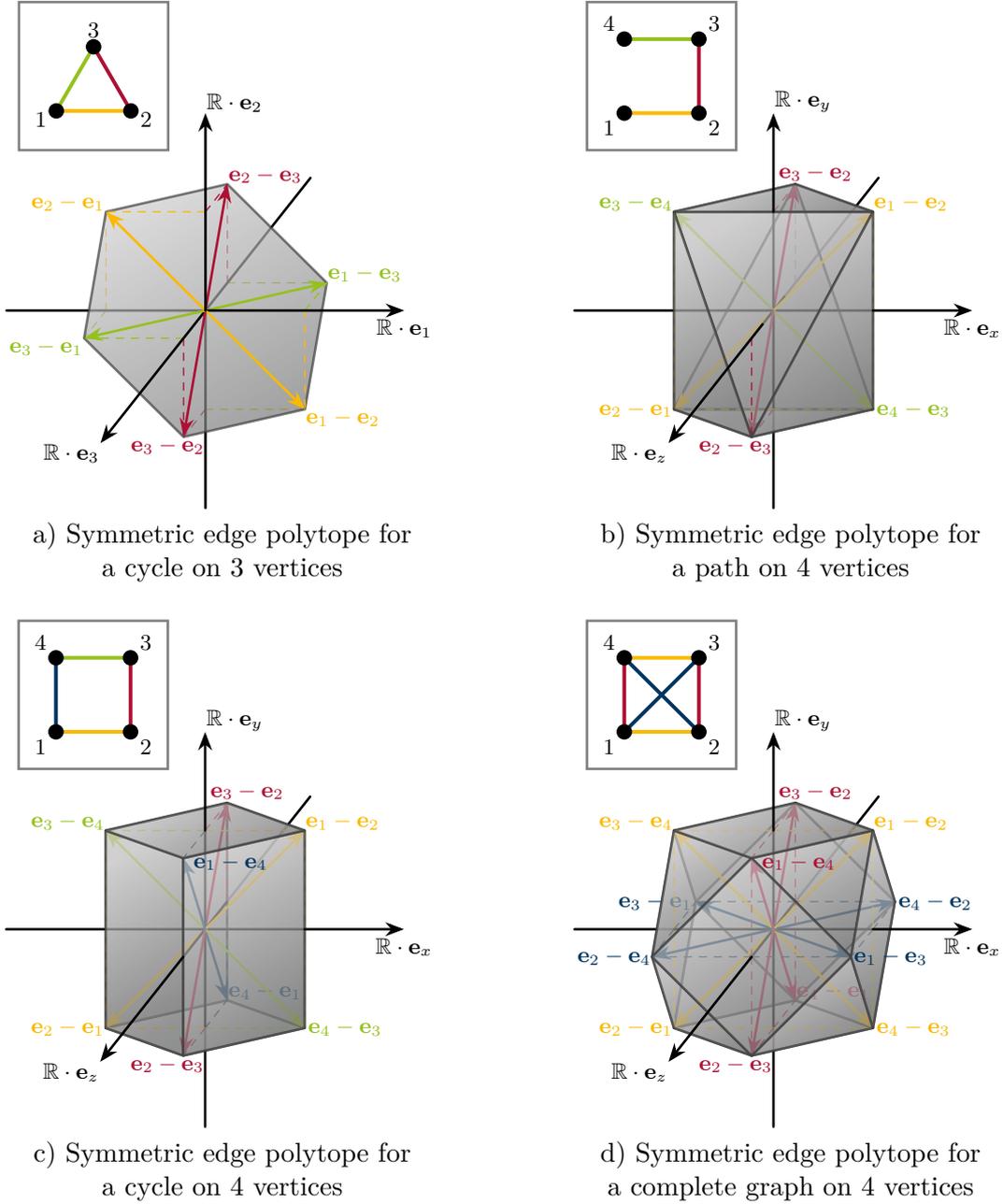
\begin{figure}%
    \tikzset{
        every picture/.style={scale=1.4, font=\footnotesize},
    	z={(-2.2mm,-2.8mm)},
    	ab-plane/.style={OSNAgold},
    	bc-plane/.style={OSNAred},
    	ca-plane/.style={RUBgreen},
    	cd-plane/.style={RUBgreen},
    	da-plane/.style={RUBblue},
    	xy-plane/.style={OSNAgold},
    	yz-plane/.style={OSNAred},
    	xz-plane/.style={RUBblue},
    	vertices/.style={line width=1pt,-Stealth,line cap=round},
    	aux-line/.style={line width=0.5pt,dashed,line cap=round},
    	coor-sys/.style={black,line width=1pt,line cap=round},
    	arrow/.style={-Stealth},
    	face/.style={shade,shading=axis,top color=RUBgray!90!black,bottom color=OSNAgray!60!black,shading angle=30,opacity=0.70},
    	polytope/.style={draw=OSNAgray!30!black,line width=1pt,line join=round},
        orig-graph/.style={shift={(-1.5,2)},scale=0.75},
        orig-vertex/.style={inner sep=2.2pt,outer sep=0pt,circle,fill=black},
        orig-arc/.style={line width=1.5pt},
        outer-box/.style={line width=1pt,gray},
    }%
    \def\axislengthneg{2}%
    \def\axislengthpos{2}%
    \def\zfactor{2.4}%
    \centering%
    \begin{minipage}{0.5\linewidth}%
        \centering%
        \begin{tikzpicture}
	\coordinate (O) at (0,0,0);

	\coordinate (e_x) at (1,0,0);
	\coordinate (e_y) at (0,1,0);
	\coordinate (e_z) at (0,0,1);

	\coordinate (e_-x) at ($-1*(e_x)$);
	\coordinate (e_-y) at ($-1*(e_y)$);
	\coordinate (e_-z) at ($-1*(e_z)$);
	
	\coordinate (e_xy) at ($(e_x)-(e_y)$);
	\coordinate (e_yz) at ($(e_y)-(e_z)$);
	\coordinate (e_zx) at ($(e_z)-(e_x)$);
	
	\coordinate (e_yx) at ($-1*(e_xy)$);
	\coordinate (e_zy) at ($-1*(e_yz)$);
	\coordinate (e_xz) at ($-1*(e_zx)$);
	
	\draw[ab-plane,aux-line] (e_-x) -- (e_yx);
	\draw[ca-plane,aux-line] (e_-x) -- (e_zx);
	\draw[ab-plane,aux-line] (e_-y) -- (e_xy);
	\draw[bc-plane,aux-line] (e_-y) -- (e_zy);
	\draw[ca-plane,aux-line] (e_-z) -- (e_xz);
	\draw[bc-plane,aux-line] (e_-z) -- (e_yz);

	\draw[coor-sys] (-\axislengthneg,0,0) -- (0,0,0);
	\draw[coor-sys] (0,-\axislengthneg,0) -- (0,0,0);
	\draw[coor-sys] (0,0,-\zfactor*\axislengthneg) -- (0,0,0);
	
	\path[polytope,face] (e_xy) -- (e_xz) -- (e_yz) -- (e_yx) -- (e_zx) -- (e_zy) -- cycle;
	
	\draw[ab-plane,aux-line] (e_x) -- (e_xy);
	\draw[ca-plane,aux-line] (e_x) -- (e_xz);
	\draw[ab-plane,aux-line] (e_y) -- (e_yx);
	\draw[bc-plane,aux-line] (e_y) -- (e_yz);
	\draw[ca-plane,aux-line] (e_z) -- (e_zx);
	\draw[bc-plane,aux-line] (e_z) -- (e_zy);
	
	\draw[ab-plane,vertices] (O) -- (e_yx) node[anchor=south east,inner sep=0pt] {$\ee_2-\ee_1$};
	\draw[bc-plane,vertices] (O) -- (e_yz) node[anchor=south west,inner sep=0pt] {$\ee_2-\ee_3$};
	\draw[ca-plane,vertices] (O) -- (e_xz) node[anchor=south west,inner sep=0pt] {$\ee_1-\ee_3$};

	\draw[coor-sys,arrow] (0,0,0) -- (\axislengthpos,0,0) node[anchor=north,inner sep=3pt] {$\mathbb{R} \cdot \ee_1$};
	\draw[coor-sys,arrow] (0,0,0) -- (0,0,\zfactor*\axislengthpos) node[anchor=north east,inner sep=0pt] {$\mathbb{R} \cdot \ee_3$};
	
	\draw[ab-plane,vertices] (O) -- (e_xy) node[anchor=north west,inner sep=0pt] {$\ee_1-\ee_2$};
	\draw[bc-plane,vertices] (O) -- (e_zy) node[anchor=30,inner sep=0pt] {$\ee_3-\ee_2$};
	\draw[ca-plane,vertices] (O) -- (e_zx) node[anchor=north east,inner sep=0pt] {$\ee_3-\ee_1$};
	
	\draw[coor-sys,arrow] (0,0,0) -- (0,\axislengthpos,0) node[anchor=south west,inner sep=0pt] {$\mathbb{R} \cdot \ee_2$};

    \begin{scope}[orig-graph]
        \begin{scope}[shift={(0,0.03)}]
            \coordinate (a) at (0,0);
            \coordinate (b) at (1,0);
            \coordinate (c) at (0.5,sqrt 3 /2);
        \end{scope}
        \draw[orig-arc,ab-plane] (a) -- (b);
        \draw[orig-arc,bc-plane] (b) -- (c);
        \draw[orig-arc,ca-plane] (c) -- (a);
        \node[name=a-node,orig-vertex] at (a) {};
        \node[name=b-node,orig-vertex] at (b) {};
        \node[name=c-node,orig-vertex] at (c) {};
        \node[anchor= 30] at (a-node) {$1$};
        \node[anchor=150] at (b-node) {$2$};
        \node[anchor=270] at (c-node) {$3$};
        \draw[outer-box] (-0.5,-0.5) -- ++(2,0) -- ++(0,2) -- ++(-2,0) -- cycle;
    \end{scope}
\end{tikzpicture}\\
        a) Symmetric edge polytope for\\
        a cycle on $3$ vertices%
    \end{minipage}%
    \begin{minipage}{0.5\linewidth}%
        \centering%
        \begin{tikzpicture}
	\coordinate (O) at (0,0,0);

	\coordinate (e_x) at (1,0,0);
	\coordinate (e_y) at (0,1,0);
	\coordinate (e_z) at (0,0,1);

	\coordinate (e_-x) at ($-1*(e_x)$);
	\coordinate (e_-y) at ($-1*(e_y)$);
	\coordinate (e_-z) at ($-1*(e_z)$);
	
	\coordinate (e_ab) at ($ (e_x)+(e_y) $);
	\coordinate (e_bc) at ($ (e_z)+(e_-y)$);
	\coordinate (e_cd) at ($(e_-x)+(e_y) $);
	\coordinate (e_da) at ($(e_-z)+(e_-y)$);
	
	\coordinate (e_ba) at ($-1*(e_ab)$);
	\coordinate (e_cb) at ($-1*(e_bc)$);
	\coordinate (e_dc) at ($-1*(e_cd)$);
	\coordinate (e_ad) at ($-1*(e_da)$);
	
	\draw[coor-sys] (-\axislengthneg,0,0) -- (-1,0,0);
	\draw[coor-sys] (0,-\axislengthneg,0) -- (0,-1,0);
	\draw[coor-sys] (0,0,-\zfactor*\axislengthneg) -- (0,0,-0.5);
	
	\path[polytope,face] (e_cd) -- (e_cb) -- (e_ba) -- cycle;
	\path[polytope,face] (e_ab) -- (e_cb) -- (e_dc) -- cycle;
	\path[polytope,face] (e_dc) -- (e_bc) -- (e_ba) -- cycle;
	
	\draw[bc-plane,aux-line] (e_-z) -- (e_cb);
	
	\path[polytope,face] (e_ba) -- (e_cb) -- (e_dc) -- cycle;
	
	\draw[bc-plane,vertices] (O) -- (e_cb) node[anchor=210,inner sep=1pt] {$\ee_3-\ee_2$};
	\draw[bc-plane,aux-line] (e_y) -- (e_cb);
	
	\draw[coor-sys] (0,-1,0) -- (0,1,0);
	\draw[coor-sys] (0,0,-0.5) -- (0,0,0);
	
	\draw[ab-plane,aux-line] (e_-x) -- (e_ba);
	\draw[ab-plane,vertices] (O) -- (e_ba) node[anchor= east,inner sep=0pt] {$\ee_2-\ee_1$};
	\draw[ab-plane,aux-line] (e_-y) -- (e_ba);
	
	\draw[cd-plane,aux-line] (e_-x) -- (e_cd);
	\draw[cd-plane,vertices] (O) -- (e_cd) node[anchor=south east,inner sep=0pt] {$\ee_3-\ee_4$};
	\draw[cd-plane,aux-line] (e_y) -- (e_cd);

		\draw[coor-sys] (-1,0,0) -- (1,0,0);
	
	\draw[cd-plane,aux-line] (e_-y) -- (e_dc);
	\draw[cd-plane,vertices] (O) -- (e_dc) node[anchor=west,inner sep=1pt] {$\ee_4-\ee_3$};
	\draw[cd-plane,aux-line] (e_x) -- (e_dc);
	
	\draw[ab-plane,aux-line] (e_y) -- (e_ab);
	\draw[ab-plane,vertices] (O) -- (e_ab) node[anchor=south west,inner sep=0pt] {$\ee_1-\ee_2$};
	\draw[ab-plane,aux-line] (e_x) -- (e_ab);
	
	\draw[coor-sys] (0,0,0) -- (0,0,0.5);
	
	\draw[bc-plane,aux-line] (e_-y) -- (e_bc);
	\draw[bc-plane,vertices] (O) -- (e_bc) node[anchor=30,inner sep=0pt] {$\ee_2-\ee_3$};
	
	\path[polytope,face] (e_ab) -- (e_bc) -- (e_cd) -- cycle;

	\draw[bc-plane,aux-line] (e_z) -- (e_bc);
	
	\path[polytope,face] (e_cd) -- (e_bc) -- (e_ba) -- cycle;
	\path[polytope,face] (e_ab) -- (e_bc) -- (e_dc) -- cycle;
	\path[polytope,face] (e_ab) -- (e_cd) -- (e_cb) -- cycle;

	\draw[coor-sys,arrow] (1,0,0) -- (\axislengthpos,0,0) node[anchor=north,inner sep=3pt] {$\mathbb{R} \cdot \ee_x$};
	\draw[coor-sys,arrow] (0,1,0) -- (0,\axislengthpos,0) node[anchor=south west,inner sep=0pt] {$\mathbb{R} \cdot \ee_y$};
	\draw[coor-sys,arrow] (0,0,0.5) -- (0,0,\zfactor*\axislengthpos) node[anchor=north east,inner sep=0pt] {$\mathbb{R} \cdot \ee_z$};

    \begin{scope}[orig-graph]
        \coordinate (a) at (0,0);
        \coordinate (b) at (1,0);
        \coordinate (c) at (1,1);
        \coordinate (d) at (0,1);
        \draw[orig-arc,ab-plane] (a) -- (b);
        \draw[orig-arc,bc-plane] (b) -- (c);
        \draw[orig-arc,cd-plane] (c) -- (d);
        \node[name=a-node,orig-vertex] at (a) {};
        \node[name=b-node,orig-vertex] at (b) {};
        \node[name=c-node,orig-vertex] at (c) {};
        \node[name=d-node,orig-vertex] at (d) {};
        \node[anchor= 45] at (a) {$1$};
        \node[anchor=135] at (b) {$2$};
        \node[anchor=225] at (c) {$3$};
        \node[anchor=315] at (d) {$4$};
        \draw[outer-box] (-0.5,-0.5) -- ++(2,0) -- ++(0,2) -- ++(-2,0) -- cycle;
    \end{scope}
\end{tikzpicture}\\
        b) Symmetric edge polytope for\\
        a path on $4$ vertices%
    \end{minipage}%
    \vspace{0.5cm}%
    \linebreak%
    \begin{minipage}{0.5\linewidth}%
        \centering%
        \begin{tikzpicture}
	\coordinate (O) at (0,0,0);

	\coordinate (e_x) at (1,0,0);
	\coordinate (e_y) at (0,1,0);
	\coordinate (e_z) at (0,0,1);

	\coordinate (e_-x) at ($-1*(e_x)$);
	\coordinate (e_-y) at ($-1*(e_y)$);
	\coordinate (e_-z) at ($-1*(e_z)$);
	
	\coordinate (e_ab) at ($ (e_x)+(e_y) $);
	\coordinate (e_bc) at ($ (e_z)+(e_-y)$);
	\coordinate (e_cd) at ($(e_-x)+(e_y) $);
	\coordinate (e_da) at ($(e_-z)+(e_-y)$);
	
	\coordinate (e_ba) at ($-1*(e_ab)$);
	\coordinate (e_cb) at ($-1*(e_bc)$);
	\coordinate (e_dc) at ($-1*(e_cd)$);
	\coordinate (e_ad) at ($-1*(e_da)$);
	
	\draw[coor-sys] (-\axislengthneg,0,0) -- (-1,0,0);
	\draw[coor-sys] (0,-\axislengthneg,0) -- (0,-1,0);
	\draw[coor-sys] (0,0,-\zfactor*\axislengthneg) -- (0,0,-1);
	
	\path[polytope,face] (e_ab) -- (e_cb) -- (e_da) -- (e_dc) -- cycle;
	\path[polytope,face] (e_da) -- (e_dc) -- (e_bc) -- (e_ba) -- cycle;
	\path[polytope,face] (e_cd) -- (e_cb) -- (e_da) -- (e_ba) -- cycle;
	
	\draw[da-plane,aux-line] (e_-z) -- (e_da);
	\draw[da-plane,vertices] (O) -- (e_da) node[anchor=south west,inner sep=0pt] {$\ee_4-\ee_1$};
	\draw[da-plane,aux-line] (e_-y) -- (e_da);
	
	\draw[bc-plane,aux-line] (e_-z) -- (e_cb);
	\draw[bc-plane,vertices] (O) -- (e_cb) node[anchor=210,inner sep=1pt] {$\ee_3-\ee_2$};
	\draw[bc-plane,aux-line] (e_y) -- (e_cb);
	
	\draw[coor-sys] (0,-1,0) -- (0,1,0);
	\draw[coor-sys] (0,0,-1) -- (0,0,0);
	
	\draw[ab-plane,aux-line] (e_-x) -- (e_ba);
	\draw[ab-plane,vertices] (O) -- (e_ba) node[anchor=east,inner sep=0pt] {$\ee_2-\ee_1$};
	\draw[ab-plane,aux-line] (e_-y) -- (e_ba);
	
	\draw[cd-plane,aux-line] (e_-x) -- (e_cd);
	\draw[cd-plane,vertices] (O) -- (e_cd) node[anchor=south east,inner sep=0pt] {$\ee_3-\ee_4$};
	\draw[cd-plane,aux-line] (e_y) -- (e_cd);

	\draw[coor-sys] (-1,0,0) -- (1,0,0);
	
	\draw[cd-plane,aux-line] (e_-y) -- (e_dc);
	\draw[cd-plane,vertices] (O) -- (e_dc) node[anchor=west,inner sep=1pt] {$\ee_4-\ee_3$};
	\draw[cd-plane,aux-line] (e_x) -- (e_dc);
	
	\draw[ab-plane,aux-line] (e_y) -- (e_ab);
	\draw[ab-plane,vertices] (O) -- (e_ab) node[anchor=south west,inner sep=0pt] {$\ee_1-\ee_2$};
	\draw[ab-plane,aux-line] (e_x) -- (e_ab);
	
	\draw[coor-sys] (0,0,0) -- (0,0,1);
	
	\draw[bc-plane,aux-line] (e_-y) -- (e_bc);
	\draw[bc-plane,vertices] (O) -- (e_bc) node[anchor=30,inner sep=0pt] {$\ee_2-\ee_3$};
	\draw[bc-plane,aux-line] (e_z) -- (e_bc);
	
	\draw[da-plane,aux-line] (e_y) -- (e_ad);
	\draw[da-plane,vertices] (O) -- (e_ad);
	\draw[da-plane,aux-line] (e_z) -- (e_ad);
	
	\path[polytope,face] (e_cd) -- (e_ad) -- (e_bc) -- (e_ba) -- cycle;
	\path[polytope,face] (e_ab) -- (e_ad) -- (e_bc) -- (e_dc) -- cycle;
	\path[polytope,face] (e_ab) -- (e_ad) -- (e_cd) -- (e_cb) -- cycle;
	
	\draw[coor-sys,arrow] (1,0,0) -- (\axislengthpos,0,0) node[anchor=north,inner sep=3pt] {$\mathbb{R} \cdot \ee_x$};
	\draw[coor-sys,arrow] (0,1,0) -- (0,\axislengthpos,0) node[anchor=south west,inner sep=0pt] {$\mathbb{R} \cdot \ee_y$};
	\draw[coor-sys,arrow] (0,0,1) -- (0,0,\zfactor*\axislengthpos) node[anchor=north east,inner sep=0pt] {$\mathbb{R} \cdot \ee_z$};
	
	\node[da-plane,anchor=175,inner sep=4pt] at (e_ad) {$\ee_1-\ee_4$};

    \begin{scope}[orig-graph]
        \coordinate (a) at (0,0);
        \coordinate (b) at (1,0);
        \coordinate (c) at (1,1);
        \coordinate (d) at (0,1);
        \draw[orig-arc,ab-plane] (a) -- (b);
        \draw[orig-arc,bc-plane] (b) -- (c);
        \draw[orig-arc,cd-plane] (c) -- (d);
        \draw[orig-arc,da-plane] (d) -- (a);
        \node[name=a-node,orig-vertex] at (a) {};
        \node[name=b-node,orig-vertex] at (b) {};
        \node[name=c-node,orig-vertex] at (c) {};
        \node[name=d-node,orig-vertex] at (d) {};
        \node[anchor= 45] at (a) {$1$};
        \node[anchor=135] at (b) {$2$};
        \node[anchor=225] at (c) {$3$};
        \node[anchor=315] at (d) {$4$};
        \draw[outer-box] (-0.5,-0.5) -- ++(2,0) -- ++(0,2) -- ++(-2,0) -- cycle;
    \end{scope}
\end{tikzpicture}\\
        c) Symmetric edge polytope for\\
        a cycle on $4$ vertices%
    \end{minipage}%
    \begin{minipage}{0.5\linewidth}%
        \centering%
        \begin{tikzpicture}
	\coordinate (O) at (0,0,0);

	\coordinate (e_x) at (1,0,0);
	\coordinate (e_y) at (0,1,0);
	\coordinate (e_z) at (0,0,1);

	\coordinate (e_-x) at ($-1*(e_x)$);
	\coordinate (e_-y) at ($-1*(e_y)$);
	\coordinate (e_-z) at ($-1*(e_z)$);
	
	\coordinate (e_ab) at ($(e_x) + (e_y)$);
	\coordinate (e_bc) at ($(e_-y)+ (e_z)$);
	\coordinate (e_cd) at ($(e_-x)+ (e_y)$);
	\coordinate (e_da) at ($(e_-y)+(e_-z)$);
	\coordinate (e_ac) at ($(e_x) + (e_z)$);
	\coordinate (e_bd) at ($(e_-x)+ (e_z)$);
	
	\coordinate (e_ba) at ($-1*(e_ab)$);
	\coordinate (e_cb) at ($-1*(e_bc)$);
	\coordinate (e_dc) at ($-1*(e_cd)$);
	\coordinate (e_ad) at ($-1*(e_da)$);
	\coordinate (e_ca) at ($-1*(e_ac)$);
	\coordinate (e_db) at ($-1*(e_bd)$);
	
	\draw[coor-sys] (-\axislengthneg,0,0) -- (-1,0,0);
	\draw[coor-sys] (0,-\axislengthneg,0) -- (0,-1,0);
	\draw[coor-sys] (0,0,-\zfactor*\axislengthneg) -- (0,0,-1);
	
	\path[polytope,face] (e_ca) -- (e_cb) -- (e_db) -- (e_da) -- cycle;
	\path[polytope,face] (e_ba) -- (e_bc) -- (e_dc) -- (e_da) -- cycle;
	\path[polytope,face] (e_ca) -- (e_cd) -- (e_bd) -- (e_ba) -- cycle;
	\path[polytope,face] (e_cd) -- (e_cb) -- (e_ca) -- cycle;
	\path[polytope,face] (e_cb) -- (e_ab) -- (e_db) -- cycle;
	\path[polytope,face] (e_ba) -- (e_ca) -- (e_da) -- cycle;
	\path[polytope,face] (e_da) -- (e_db) -- (e_dc) -- cycle;
	
	\draw[yz-plane,aux-line] (e_-z) -- (e_da);
	\draw[yz-plane,vertices] (O) -- (e_da) node[anchor=south west,inner sep=0pt] {$\ee_4-\ee_1$};
	\draw[yz-plane,aux-line] (e_-y) -- (e_da);
	
	\draw[xz-plane,aux-line] (e_-z) -- (e_ca);
	\draw[xz-plane,vertices] (O) -- (e_ca) node[anchor=east,inner sep=1pt] {$\ee_3-\ee_1$};
	\draw[xz-plane,aux-line] (e_-x) -- (e_ca);
	
	\draw[xz-plane,aux-line] (e_-z) -- (e_db);
	\draw[xz-plane,vertices] (O) -- (e_db) node[anchor=west,inner sep=1pt] {$\ee_4-\ee_2$};
	\draw[xz-plane,aux-line] (e_x) -- (e_db);
	
	\draw[yz-plane,aux-line] (e_-z) -- (e_cb);
	\draw[yz-plane,vertices] (O) -- (e_cb) node[anchor=210,inner sep=1pt] {$\ee_3-\ee_2$};
	\draw[yz-plane,aux-line] (e_y) -- (e_cb);
	
	\draw[coor-sys] (0,-1,0) -- (0,1,0);
	\draw[coor-sys] (0,0,-1) -- (0,0,0);
	
	\draw[xy-plane,aux-line] (e_-x) -- (e_ba);
	\draw[xy-plane,vertices] (O) -- (e_ba) node[anchor=east,inner sep=0pt] {$\ee_2-\ee_1$};
	\draw[xy-plane,aux-line] (e_-y) -- (e_ba);
	
	\draw[xy-plane,aux-line] (e_-x) -- (e_cd);
	\draw[xy-plane,vertices] (O) -- (e_cd) node[anchor=south east,inner sep=0pt] {$\ee_3-\ee_4$};
	\draw[xy-plane,aux-line] (e_y) -- (e_cd);

	\draw[coor-sys] (-1,0,0) -- (1,0,0);
	
	\draw[xy-plane,aux-line] (e_-y) -- (e_dc);
	\draw[xy-plane,vertices] (O) -- (e_dc) node[anchor=west,inner sep=1pt] {$\ee_4-\ee_3$};
	\draw[xy-plane,aux-line] (e_x) -- (e_dc);
	
	\draw[xy-plane,aux-line] (e_y) -- (e_ab);
	\draw[xy-plane,vertices] (O) -- (e_ab) node[anchor=south west,inner sep=0pt] {$\ee_1-\ee_2$};
	\draw[xy-plane,aux-line] (e_x) -- (e_ab);
	
	\draw[coor-sys] (0,0,0) -- (0,0,1);
	
	\draw[xz-plane,aux-line] (e_-x) -- (e_bd);
	\draw[xz-plane,vertices] (O) -- (e_bd) node[anchor=east,inner sep=0pt] {$\ee_2-\ee_4$};
	\draw[xz-plane,aux-line] (e_z) -- (e_bd);
	
	\draw[xz-plane,aux-line] (e_x) -- (e_ac);
	\draw[xz-plane,vertices] (O) -- (e_ac);
	\draw[xz-plane,aux-line] (e_z) -- (e_ac);
	
	\draw[yz-plane,aux-line] (e_-y) -- (e_bc);
	\draw[yz-plane,vertices] (O) -- (e_bc) node[anchor=30,inner sep=0pt] {$\ee_2-\ee_3$};
	\draw[yz-plane,aux-line] (e_z) -- (e_bc);
	
	\draw[yz-plane,aux-line] (e_y) -- (e_ad);
	\draw[yz-plane,vertices] (O) -- (e_ad);
	\draw[yz-plane,aux-line] (e_z) -- (e_ad);
	
	\path[polytope,face] (e_dc) -- (e_bc) -- (e_ac) -- cycle;
	\path[polytope,face] (e_bc) -- (e_ba) -- (e_bd) -- cycle;
	\path[polytope,face] (e_ab) -- (e_ac) -- (e_ad) -- cycle;
	\path[polytope,face] (e_ad) -- (e_bd) -- (e_cd) -- cycle;
	\path[polytope,face] (e_ac) -- (e_bc) -- (e_bd) -- (e_ad) -- cycle;
	\path[polytope,face] (e_ab) -- (e_cb) -- (e_cd) -- (e_ad) -- cycle;
	\path[polytope,face] (e_ac) -- (e_dc) -- (e_db) -- (e_ab) -- cycle;
	
	\draw[coor-sys,arrow] (1,0,0) -- (\axislengthpos,0,0) node[anchor=north,inner sep=3pt] {$\mathbb{R} \cdot \ee_x$};
	\draw[coor-sys,arrow] (0,1,0) -- (0,\axislengthpos,0) node[anchor=south west,inner sep=0pt] {$\mathbb{R} \cdot \ee_y$};
	\draw[coor-sys,arrow] (0,0,1) -- (0,0,\zfactor*\axislengthpos) node[anchor=north east,inner sep=0pt] {$\mathbb{R} \cdot \ee_z$};
	
	\node[xz-plane,anchor=west,inner sep=1pt] at (e_ac) {$\ee_1-\ee_3$};
	\node[yz-plane,anchor=175,inner sep=4pt] at (e_ad) {$\ee_1-\ee_4$};
    
    \begin{scope}[orig-graph]
        \coordinate (a) at (0,0);
        \coordinate (b) at (1,0);
        \coordinate (c) at (1,1);
        \coordinate (d) at (0,1);
        \draw[orig-arc,xy-plane] (a) -- (b);
        \draw[orig-arc,xy-plane] (c) -- (d);
        \draw[orig-arc,yz-plane] (b) -- (c);
        \draw[orig-arc,yz-plane] (d) -- (a);
        \draw[orig-arc,xz-plane] (a) -- (c);
        \draw[orig-arc,xz-plane] (b) -- (d);
        \node[name=a-node,orig-vertex] at (a) {};
        \node[name=b-node,orig-vertex] at (b) {};
        \node[name=c-node,orig-vertex] at (c) {};
        \node[name=d-node,orig-vertex] at (d) {};
        \node[anchor= 45] at (a) {$1$};
        \node[anchor=135] at (b) {$2$};
        \node[anchor=225] at (c) {$3$};
        \node[anchor=315] at (d) {$4$};
        \draw[outer-box] (-0.5,-0.5) -- ++(2,0) -- ++(0,2) -- ++(-2,0) -- cycle;
    \end{scope}
\end{tikzpicture}\\
        d) Symmetric edge polytope for\\
        a complete graph on $4$ vertices%
    \end{minipage}%
    \caption{%
        Examples of symmetric edge polytopes for different graphs; a) shows a polytope in $\RR^3$, depicted in its natural embedding. The polytopes in b)--d) are subsets of $\RR^4$, but each lies entirely within the hyperplane orthogonal to $\ee_1 + \ee_2 + \ee_3 + \ee_4$. Only this hyperplane is visualized, represented via the basis vectors
        $\ee_x=\frac12 (\ee_1 - \ee_2 - \ee_3 + \ee_4)$,
        $\ee_y=\frac12 (\ee_1 - \ee_2 + \ee_3 - \ee_4)$,
        $\ee_z=\frac12 (\ee_1 + \ee_2 - \ee_3 - \ee_4)$.%
    }
    \label{fig:SEPex}
\end{figure}

To introduce randomness into this construction, we take \(G\) to be a realization of an Erd\H{o}s--R\'enyi random graph \(G_{n,p}\)  
(see \Cref{sec:Statement} for the definition and \cite{BollobasRandomGraphs,JansonRandomGraphs,VanDerHofstad} for background material on this random graph model).
We then consider the associated random symmetric edge polytope
\[
\SEP_{n,p} \coloneqq \SEP_{G_{n,p}} .
\]
This random polytope inherits intricate combinatorial properties from the underlying random graph \(G_{n,p}\), which in turn determine its geometry and face structure.

Our primary focus lies on understanding the number of edges of \(\SEP_{n,p}\), as well as the number of edges appearing in any of its unimodular triangulations (which exist due to \cite{Arithmetic_apects_SEP}), see \Cref{sec:BackgroundPolytopes} for formal definitions. 
We demonstrate that these random variables exhibit a rich and previously unexplored asymptotic behavior as \(n \to \infty\).
Using a detailed combinatorial-geometric analysis of the underlying graph configurations, we derive sharp asymptotics for the expectations and variances of both the polytope edge count and the number of edges in any unimodular triangulation.
This analysis reveals a non-generic parameter regime in which the leading variance term is suppressed, giving rise to a slower growth of fluctuations, an effect that has no apparent analogue in classical subgraph counting problems for Erd\H{o}s--R\'enyi random graphs.
Building on these structural results, we apply the discrete Malliavin--Stein method for normal approximation to establish central limit theorems with explicit rates of convergence for the normalized edge counts.

Previously, for random symmetric edge polytopes, only the number of facets and concentration inequalities for face numbers had been investigated, see \cite{BraunSEP,gammavectorsymmetricedgepoyltope}.
To the best of our knowledge, the present work provides the first distributional limit theorems for random lattice polytopes and the first systematic study of their high-dimensional fluctuations.

\subsection{Statement and discussion of the main result}\label{sec:Statement}

Let \(G_{n,p}\) denote the Erd\H{o}s--R\'enyi random graph on \(n\) nodes, where each arc of the complete graph is retained independently with probability \(p = p(n) \in (0,1)\). Throughout, we allow \(p\) to depend on \(n\) and suppress this dependence in the notation when unambiguous. Our main objects of interest are random variables counting the number of edges of the random symmetric edge polytope \(\SEP_{n,p}\) associated with $G_{n,p}$, and the number of edges appearing in any unimodular triangulation of \(\SEP_{n,p}\). We denote any such random variable by \(K_{n,p}\) and determine its expectation and variance. Furthermore, we establish a central limit theorem for \( K_{n,p} \) as $n\to\infty$. To describe the asymptotic behaviour of these and related quantities, we use the notation introduced in the following table.

\begin{table}[h]%
	\renewcommand{\arraystretch}{1.2}%
    \centering%
	\begin{tabular}{ccl}%
		\toprule%
		\bfseries Notation & \bfseries Definition & \bfseries
            \begin{minipage}{8.4em}%
                \centering%
                Landau symbol%
            \end{minipage}\\
		\midrule%
		$a_n\simeq b_n$
		& $0<\liminf_{n\to \infty}\lvert\frac{a_n}{b_n}\rvert\leq \limsup_{n\to \infty} \lvert\frac{a_n}{b_n}\rvert < \infty$
		& \hspace{3em}$a_n\in \Theta(b_n)$
		\\
		$a_n\ll b_n$
		& \hspace{6em}$\mathllap{\limsup_{n\to \infty} \lvert\frac{a_n}{b_n}\rvert} = \mathrlap{0}$
		& \hspace{3em}$a_n\in o(b_n)$
		\\
		$a_n\gg b_n$
		& \hspace{6em}$\mathllap{\liminf_{n\to \infty} \lvert\frac{a_n}{b_n}\rvert} = \mathrlap{\infty}$
		& \hspace{3em}$a_n\in \omega(b_n)$
		\\
		$a_n \lesssim b_n$
		& \hspace{6em}$\mathllap{\limsup_{n\to \infty} \lvert\frac{a_n}{b_n}\rvert} < \mathrlap{\infty}$
		& \hspace{3em}$ a_n\in O(b_n)$
		\\
		$a_n \gtrsim b_n$
		& \hspace{6em}$\mathllap{\liminf_{n\to \infty} \lvert\frac{a_n}{b_n}\rvert} > \mathrlap{0}$
		& \hspace{3em}$a_n\in \Omega(b_n)$
		\\
		\bottomrule%
	\end{tabular}%
	\caption{%
        Notation for asymptotic relations between two sequences
        $(a_n)_{n\geq0}$~and~$(b_n)_{n\geq 0}$
        with the notation in the first column adopted throughout the paper
        and the classical Landau symbol included for reference.%
    }%
    \label{tab:notation}%
\end{table}%

Having set up these notations, we can now formulate our main result. 
\begin{thm}\label{thm:Main}
    Assume that \( p \gg n^{-2} \) and \( \limsup_{n\to\infty} p < 1 \).
    Let \( K_{n,p} \) denote either
    \begin{itemize}
        \item the number of edges of the random symmetric edge polytope \( \SEP_{n,p} \), or
        \item the number of edges appearing in any unimodular triangulation of \( \SEP_{n,p} \).
    \end{itemize}
    If \(K_{n,p}\) corresponds to the edge count of \( \SEP_{n,p} \), we additionally assume that \(1/\sqrt{2}\) is not an accumulation point of \(p\).
    Then, as \( n \to \infty \), the following statements hold:
    \begin{enumerate}
    \item\label{thm:Main:a} The expectation and variance satisfy
        \[
        \EE[K_{n,p}] \simeq n^4 p^2
        \qquad \text{and} \qquad
        \var(K_{n,p}) \simeq n^6 p^3 .
        \]
    \item\label{thm:Main:b} The normalized random variable satisfies a central limit theorem, that is,
        \[
        \frac{K_{n,p} - \EE[K_{n,p}]}{\sqrt{\var(K_{n,p})}}
        \xrightarrow{d} \mathcal{N}(0,1),
        \]
        where \( \xrightarrow{d} \mathcal{N}(0,1)\) denotes convergence in distribution to a standard Gaussian random variable.
    \end{enumerate}
\end{thm}

The preceding theorem summarizes the principal asymptotic results of this paper in a compact form. The assumptions on the edge probability \(p\) are chosen to ensure that both the expectation and the variance of \(K_{n,p}\) diverge and that a nondegenerate Gaussian limit emerges. In the subsequent sections, we derive more refined asymptotic expansions for the expectation and the variance, and establish a quantitative central limit theorem with explicit rates of convergence. The following remark places these results into context and highlights the non-generic parameter regime that requires separate consideration.

\begin{rem}
\begin{enumerate}
\item
    The expectation asymptotics stated in \Cref{thm:Main} capture the leading-order behavior under the assumption that \( \limsup_{n\to\infty} p < 1 \). A more refined analysis, valid for the full range of parameters \(p\), is given in \Cref{expectation_sep} for the polytope edge count and in \Cref{lem:ExpectationTriangulation} for the number of edges in a unimodular triangulation. In both cases,
    \(
    \EE[K_{n,p}] \to \infty
    \)
    if and only if \( p \gg n^{-2} \), identifying the natural sparsity threshold for the emergence of nondegenerate fluctuations and justifying the lower bound on \(p\) imposed in the theorem.
    
\item
    For the variance, \Cref{thm:Main} states the generic order of magnitude under the stated assumptions. A detailed asymptotic expansion for the variance of the polytope edge count, valid for all admissible values of \(p\), is derived in \Cref{variance_sep}. This refinement shows that the leading term has the form $n^6 p^3 (1-p) \bigl( 1 - \sqrt{2}p \bigr)^2$ and vanishes as \(p\) approaches \(1/\sqrt{2}\), resulting in a slower growth of the variance. This type of degeneracy has no apparent analogue in classical subgraph counting problems in Erd\H{o}s--R\'enyi random graphs and is therefore not predicted by existing results in random graph theory. In contrast, for unimodular triangulations, \Cref{variance_sepTriag} provides a uniform lower bound on the variance in which this degeneracy phenomenon is not present.
    
\item
    \Cref{CLT_SymmetricEdge} establishes a quantitative central limit theorem for the polytope edge count, providing explicit rates of convergence in the Kolmogorov distance. Away from the critical value \(p = 1/\sqrt{2}\), the dominant term in the bound is of order \( (n\sqrt{p(1-p)})^{-1} \), comparable to rates obtained for classical subgraph counting problems in Erd\H{o}s--R\'enyi random graphs. As \(p\) approaches \(1/\sqrt{2}\), the rate deteriorates in accordance with the refined variance asymptotics in \Cref{variance_sep}. For unimodular triangulations, \Cref{thm:CLTtriangulations} yields a central limit theorem with a simpler rate, reflecting the absence of a corresponding variance suppression phenomenon in this setting.
\end{enumerate}
\end{rem}

The proof of \Cref{thm:Main} combines a detailed combinatorial-geometric analysis with probabilistic normal approximation. The expectation and variance asymptotics in part~\ref{thm:Main:a} are obtained by expanding the relevant edge counts into sums of indicator variables and by enumerating the contributing configurations in the underlying random graph. In particular, the variance is analyzed via a covariance decomposition in which terms are grouped according to the overlap pattern of the corresponding arcs, leading to the refined expansions in \Cref{expectation_sep,variance_sep} (and their triangulation analogues, \Cref{lem:ExpectationTriangulation,variance_sepTriag}). The special role of \(p=1/\sqrt{2}\) is already visible at the level of the covariance enumeration, where contributions of order \(n^6p^3(1-p)\) combine into a factor \((1-\sqrt{2}p)^2\), leading to a suppression of the leading variance term at $p=1/\sqrt{2}$. Part~\ref{thm:Main:b}, and in particular the quantitative bounds established later on, rely on the discrete Malliavin--Stein method for normal approximation on product spaces. The edge counts under consideration can be expressed as non-linear functions of independent Rademacher variables encoding the edge set of the Erd\H{o}s--R\'enyi graph. This allows us to study fluctuations through first- and second-order discrete gradients and to control the Kolmogorov distance to a Gaussian law via a discrete second-order Poincar\'e inequality, see \Cref{prop:2ndOrderPoincare}.

 These arguments are tailored to the case of edges and already illustrate both the combinatorial richness and the probabilistic structure of the model. It is therefore natural to ask to what extent our approach can be extended beyond this setting. A natural extension of our work would be to ask for \(k\)-faces of the symmetric edge polytope or of its unimodular triangulations for \(k\geq 2\). However, this would first require an explicit description of the \(k\)-faces of a symmetric edge polytope in terms of the graph, since such a description is not yet available. Although \(k\)-faces of unimodular triangulations admit graph-theoretic descriptions, these become substantially more complicated for \(k\geq 2\) than in the edge case, involving cycles of larger length, see \cite[Proposition 3.8]{Arithmetic_apects_SEP}. Another natural generalization would be to replace the Erd\H{o}s--R\'enyi graph by other random graph models. One possible example is the random geometric graph, where one samples random points in the plane and connects two of them whenever they are sufficiently close.

\section{Preliminaries}

The goal of this section is to provide the necessary background from probability theory and polytope theory. Readers who are familiar with one of these areas may wish to skip the corresponding subsection. Nevertheless, in order to make the article accessible to researchers from both communities and to keep the presentation self-contained, we include background material from both fields.

\subsection{Background from probability theory}\label{sec:backgroundprob}

Let \( n \geq 1 \) and let \( p_1, \ldots, p_n \in (0,1) \). We fix a probability space $(\Omega,\mathscr{A},\PP)$ that is rich enough to support all random objects considered throughout this paper. Let \( X_1, \ldots, X_n \) be independent \textit{Rademacher random variables} satisfying
\[
\PP(X_i = 1) = p_i \qquad \text{and} \qquad \PP(X_i = -1) = q_i
\]
for each \( i \in [n] \), where $q_i=1 - p_i
$. Let \( F = F(X_1, \ldots, X_n) \) be a real-valued, measurable function of these variables. For \( i \in [n] \), we define
\begin{align*}
    F^+_i &\coloneqq F(X_1, \ldots, X_{i-1}, +1, X_{i+1}, \ldots, X_n),\\
    F^-_i &\coloneqq F(X_1, \ldots, X_{i-1}, -1, X_{i+1}, \ldots, X_n).
\end{align*}
The \emph{discrete gradient} \( \D_i F \) of \( F \) at the \( i \)\textsuperscript{th} coordinate is defined as
\[
\D_i F \coloneqq \sqrt{p_iq_i} (F^+_i - F^-_i).
\]
Up to the scaling factor \( \sqrt{p_iq_i} \), this quantity measures the sensitivity of \( F \) to flipping the \( i \)\textsuperscript{th} variable from \( +1 \) to \( -1 \), or vice versa.

We will also require an iterated version of this operator. To simplify notation, for \( i,j \in [n] \) we define
\begin{align*}
    F_{j,i}^{+,+} \coloneqq (F_j^+)_i^+, \qquad
    F_{j,i}^{+,-} \coloneqq (F_j^+)_i^-, \qquad
    F_{j,i}^{-,+} \coloneqq (F_j^-)_i^+, \qquad
    F_{j,i}^{-,-} \coloneqq (F_j^-)_i^-.
\end{align*}
The \emph{second-order discrete gradient} $\D_i\D_j F$ can then be written as
\begin{align*}
    \D_i\D_j F
    \coloneqq \D_i ( \D_j F )
    &= \sqrt{p_iq_i} \bigl(
            (\D_j F)^+_i - (\D_j F)^-_i
        \bigr)
        \\
    &= \sqrt{p_iq_i p_jq_j} \bigl(
              F_{j,i}^{+,+}
            - F_{j,i}^{-,+}
            - F_{j,i}^{+,-}
            + F_{j,i}^{-,-}
        \bigr).
\end{align*}
Note that if $i=j$, then $\D_i\D_j F = 0$ by definition.

The discrete gradients \( \D_i F \) and the second-order discrete gradients $\D_i\D_j F$ for $ i, j \in [n] $
are essential tools for quantifying the proximity of the random variable \( F \) to a Gaussian distribution. To measure the distance between two real-valued random variables \( X \) and \( Y \), assumed to be defined on the common probability space $(\Omega,\mathscr{A},\PP)$, we use the \emph{Kolmogorov distance}, defined by
\[
d_K(X,Y) \coloneqq \sup_{z \in \RR} \left| \PP(X \leq z) - \PP(Y \leq z) \right|.
\]
The following result is a simplified version of Theorem 4.1 in \cite{zgstheorem} and serves as the main technical tool for establishing our central limit theorems. It was obtained using the discrete Malliavin--Stein method on product probability spaces as developed in \cite{KRT16aihp,KRT17aop,NPR10ejp}. In this framework, the discrete gradients \( \D_i F \) play the role of Malliavin derivatives, measuring the local sensitivity of \( F \) with respect to the underlying Rademacher variables, while second-order discrete gradients capture interactions between different coordinates, see \cite{Privault09lnm}. Stein’s method for normal approximation then allows one to translate quantitative control of these derivatives into explicit bounds on distributional distances, such as the Kolmogorov distance. Second-order Poincaré-type inequalities of this form are by now a well-established and powerful tool for proving quantitative central limit theorems in discrete settings. 

\begin{proposition}\label{prop:2ndOrderPoincare}
    Let \( F = F(X_1, \ldots, X_n) \) be a real-valued, measurable function of \( n \geq 1 \) independent Rademacher random variables \( X_1, \ldots, X_n \), with \( \PP(X_i = +1)=p_i\) and \(\PP(X_i = -1)=q_i \) for all \( i \in [n] \). Assume that \( \EE[F] = 0 \) and \( \EE[F^2] = 1 \), and define
    \begin{align*}
        B_1 &\coloneqq \sum_{j,k,\ell=1}^n \sqrt{ \EE[ (\D_j F)^2 (\D_k F)^2 ] } \cdot \sqrt{ \EE[ (\D_\ell \D_j F)^2 (\D_\ell \D_k F)^2 ] }, \\
        B_2 &\coloneqq \sum_{j,k,\ell=1}^n \frac{1}{p_\ell q_\ell} \, \EE[ (\D_\ell \D_j F)^2 (\D_\ell \D_k F)^2 ], \\
        B_3 &\coloneqq \sum_{k=1}^n \frac{1}{p_k q_k} \, \EE[ (\D_k F)^4 ], \\
        B_4 &\coloneqq \sum_{k,\ell=1}^n \frac{1}{p_k q_k}  \sqrt{ \EE[ (\D_k F)^4 ] } \cdot \sqrt{ \EE[ (\D_\ell \D_k F)^4 ] }, \\
        B_5 &\coloneqq \sum_{k,\ell=1}^n \frac{1}{p_k q_k p_\ell q_\ell} \, \EE[ (\D_\ell \D_k F)^4 ].
    \end{align*}
    Then, the Kolmogorov distance between \( F \) and a standard Gaussian random variable \( N \sim \mathcal{N}(0,1) \) satisfies
    \begin{align*}
        d_K(F, N)
        &\leq \frac{\sqrt{15}}{2} \sqrt{B_1} + \frac{\sqrt{3}}{2} \sqrt{B_2} + 2 \sqrt{B_3} + 2 \sqrt{6} \sqrt{B_4} + 2 \sqrt{3} \sqrt{B_5}.
    \end{align*}
\end{proposition}

\begin{rem}
    We present here a version of Theorem 4.1 in \cite{zgstheorem} restricted to functions \( F \) depending on only \textit{finitely} many Rademacher random variables. In this setting, the additional technical assumptions required in \cite{zgstheorem} are automatically satisfied.
\end{rem}

Suppose we are given a sequence \( (F_n)_{n \geq 1} \) of random variables, where each \( F_n \) depends on \( n \) independent Rademacher random variables $X_1,\ldots,X_n$ as above. Then \Cref{prop:2ndOrderPoincare} implies that, in order to establish a central limit theorem for \( F_n \) as $n\to\infty$, it suffices to control the first- and second-order discrete gradients and to verify that the quantities \( B_1, \ldots, B_5 \), which now depend on $n$, converge to zero. Moreover, the proposition provides an explicit rate of convergence in Kolmogorov distance, which we exploit in the derivation of our main results.

\subsection{Background from polytope theory}\label{sec:BackgroundPolytopes}

In this section, we provide the necessary background on polytopes. For more details, we refer the reader to \cite{Beck} and \cite{Ziegler}. 

A \emph{lattice polytope} \( P \subset \RR^d \) is the convex hull of finitely many points in \( \ZZ^d \), that is,
\[
P = \conv(x_1, \dots, x_n), \qquad \text{for } x_1, \dots, x_n \in \ZZ^d.
\]
A \emph{supporting hyperplane} of \( P \) is a hyperplane such that \( P \) is contained in exactly one of its two closed halfspaces. A \emph{face} of \( P \) is the intersection of \( P \) with a supporting hyperplane. Throughout this paper, we do not regard \( P \) itself as a face. The \emph{dimension} of a face is defined as the dimension of its affine hull. Faces of dimension zero and one are called \emph{vertices} and \emph{edges}, respectively.

A \emph{(lattice) triangulation} of a $d$-dimensional lattice polytope \( P \) is a subdivision of \( P \) into $d$-dimensional lattice simplices such that their union equals \( P \), and any two simplices intersect in a (possibly empty) common face. A triangulation is called \emph{unimodular} if all of its $d$-simplices have volume
\(\tfrac{1}{d!}\)
with respect to the usual Lebesgue measure, and hence are lattice simplices of minimal possible volume. Equivalently, each simplex is lattice equivalent to the standard simplex spanned by the origin together with the standard unit vectors \(\ee_i\), \(1 \leq i \leq d\). If a unimodular triangulation of \( P \) exists, then necessarily every lattice point of \( P \) appears as a vertex of at least one simplex of the triangulation. Any triangulation \( \mathcal{T} \) naturally gives rise to a simplicial complex \( \Delta_{\mathcal{T}} \) consisting of all faces of the simplices in \( \mathcal{T} \), that is,
\[
\Delta_{\mathcal{T}}
= \{ F~ : ~F \text{ is a face of } G \text{ for some } G \in \mathcal{T} \} \cup \mathcal{T}.
\]
In the following, we do not always distinguish between the triangulation \( \mathcal{T} \) and the associated simplicial complex \( \Delta_{\mathcal{T}} \). In particular, when referring to $k$-dimensional faces (for example, edges) of \( \mathcal{T} \), we implicitly mean the corresponding faces of \( \Delta_{\mathcal{T}} \). Although the number of $k$-dimensional faces of a triangulation generally depends on the chosen triangulation, it is known that all unimodular triangulations of a given lattice polytope~--~if they exist~--~have the same number of $k$-dimensional faces, see, for instance, \cite{Sturmfels}. Consequently, the number of edges of a unimodular triangulation of \( P \) is a well-defined quantity.

Given a graph \( G = (V,E) \), we recall that the symmetric edge polytope associated to \( G \) is defined by
\[
\SEP_G \coloneqq \conv\bigl\{ \pm (\ee_i - \ee_j)~ :~ \{i,j\} \in E \bigr\} \subset \RR^{V},
\]
where \( (\ee_i)_{i \in V} \) denotes the standard basis of \( \RR^{V} \). For each arc \( \{i,j\} \in E \), the polytope \( \SEP_G \) contains the two vertices
\[
\ee_{i,j} \coloneqq \ee_i - \ee_j
\quad \text{and} \quad
\ee_{j,i} \coloneqq \ee_j - \ee_i,
\]
which we refer to as an \emph{antipodal pair}. We interpret the vertex \( \ee_{i,j} \) as the directed arc \( i \to j \), thereby identifying the vertices of \( \SEP_G \) with the directed arcs of the underlying graph \( G \). From this perspective, an undirected arc of \( G \) gives rise to a pair  of directed arcs, each corresponding to a vertex of $\SEP_G$.

The following result, taken from \cite[Theorem 2.2]{gammavectorsymmetricedgepoyltope}, provides a purely combinatorial characterization of the edges of a symmetric edge polytope in terms of the subgraph structure of the underlying graph. This characterization is especially well suited for probabilistic analysis when \( G \sim G_{n,p} \).

\begin{proposition}\label{prop:EdgeCharSEP}
    Let \( G \) be a graph. Two vertices of \( \SEP_G \) are connected by an edge if and only if they are not antipodal and there exists no directed $3$- or $4$-cycle in \( G \) containing both of the corresponding directed arcs.
\end{proposition}

We remark that if \( G \) consists of a single arc, then \( \SEP_G \) is a line segment and hence one-dimensional. In particular, since we do not consider a polytope to be a face of itself, the symmetric edge polytope of such a graph does not have any edges.

Beyond the structure of \( \SEP_G \) itself, we deal with unimodular triangulations of symmetric edge polytopes. It was shown in \cite[Proposition 3.8]{Arithmetic_apects_SEP}, via the construction of a squarefree Gr\"obner basis, that \( \SEP_G \) admits a unimodular triangulation. The triangulation produced there depends on the choice of a total order on the arcs of \( G \). From this construction, one can extract a precise combinatorial description of the edges appearing in a specific unimodular triangulation.

\begin{proposition}\label{edges_SEP_triag}
    Let \( G \) be a graph and let \( \prec \) be a total order on the arcs of \( G \).
    Then there exists a unimodular triangulation \( \mathcal{T} \) of \( \SEP_G \), on the vertices of $\SEP_G$ and the origin, such that two vertices \( \mathbf{x} \) and \( \mathbf{y} \) of \( \mathcal{T} \) form an edge of \( \mathcal{T} \) if and only if one of the following conditions holds:
    \begin{conditions}
        \item\label{cond:origin} one of \( \mathbf{x} \) and \( \mathbf{y} \) is the origin \( \mathbf{0} \); or
        \item there is no directed $3$-cycle in \( G \) containing the directed arcs corresponding to \( \mathbf{x} \) and \( \mathbf{y} \), and if there exists a directed $4$-cycle containing these arcs, then the minimal arc of this cycle with respect to \( \prec \) is one of the arcs corresponding to \( \mathbf{x} \) or \( \mathbf{y} \).
    \end{conditions}
\end{proposition}

\Cref{cond:origin} reflects the fact that the triangulation $\mathcal{T}$ is a cone over a triangulation of the boundary of $P$. We note that this can always be achieved for symmetric edge polytopes since those are known to be reflexive. This structural description will be crucial for our analysis of the number of edges in unimodular triangulations of random symmetric edge polytopes \( \SEP_{n,p} \).

\section[Proof of Theorem 1.1 for symmetric edge polytopes]{Proof of \Cref{thm:Main} for symmetric edge polytopes}\label{sec:SEP-proof}

In this section, we present the proof of \Cref{thm:Main} in the case where the random variable \(K_{n,p}\) is the number of edges of the random symmetric edge polytope \(\SEP_{n,p}\) arising from the Erd\H{o}s--R\'enyi random graph \(G_{n,p}\). Our goal is to establish both, the asymptotic formulas for \(\EE[K_{n,p}]\) and \(\VV(K_{n,p})\), and to prove that the appropriately normalized random variable \(K_{n,p}\) converges in distribution to a standard Gaussian one.

\subsection{Expectation}
We begin by deriving an exact expression for \(\EE[K_{n,p}]\), the expected number of edges of the random symmetric edge polytope \(\SEP_{n,p}\). 

\begin{thm}\label{expectation_sep}
    Let \(K_{n,p}\) denote the number of edges of \(\SEP_{n,p}\).
    Then,
    \[
        \EE[K_{n,p}]
        \;=\;
            12 \binom{n}{4} \, p^2 \bigl(1 - p^2\bigr)
        \;+\;
            6 \binom{n}{3} \, p^2 \Bigl( 1 + (1 - p)\,\bigl(1 - p^2\bigr)^{n-3} \Bigr)
        .
    \]
    As \(n \to \infty\), this yields
    \[
        \EE[K_{n,p}]
        \;\simeq\; n^4 p^2 (1-p) + n^3 p^2
        \;\simeq \; \begin{cases}
            n^3 p^2, & \text{if } 1-p \ll n^{-1}, \\
            n^4 p^2 (1-p), & \text{if } 1-p \gtrsim n^{-1}.
        \end{cases}
    \]
    In particular, we have
    \( \EE[K_{n,p}] \;\to\; \infty\)
    if and only if
    \( p \gg n^{-2} \).
\end{thm}

\begin{proof}
    To compute \(\EE[K_{n,p}]\), we count pairs of (undirected) arcs \(\{e,f\}\) in \(G_{n,p}\) that, after suitable orientation, form an edge of \(\SEP_{n,p}\).
    Recall from \Cref{prop:EdgeCharSEP} that an edge is realized if and only if the corresponding orientations of $e$ and $f$ are not an antipodal pair and do not lie on a common directed $3$- or $4$-cycle.
    We distinguish the following cases:
    
    \begin{maincases}
    
    \item[\(e \cap f = \emptyset\)]\label{lem:exp-case:empty}
        In this scenario, \(e\) and \(f\) are disjoint arcs. There are \(3 \binom{n}{4}\) ways to choose the unordered pair \(\{e,f\}\) on four disjoint nodes. For each choice, there are four ways to orient \(e\) and \(f\).
        For these oriented arcs to violate the edge condition of \(\SEP_{n,p}\) in \Cref{prop:EdgeCharSEP}, there must be a directed 4-cycle containing both.
        Crucially, once the orientations of \(e\) and \(f\) are fixed, there is exactly one cyclic ordering of the four nodes that respects these directions. Consequently, forming a directed 4-cycle requires the presence of exactly two specific additional arcs in \(G_{n,p}\).
        The arcs \(e\) and \(f\) themselves occur with probability \(p^2\), and the 4-cycle is avoided with probability \(1-p^2\). Hence, the total contribution from disjoint arcs is
        \[
        12 \binom{n}{4} p^2 (1 - p^2).
        \]
    
    \item[\(e\cap f = \{v\}\) for some node \(v \in V\)]
        There are \(\binom{n}{3}\) ways to choose the three involved nodes and three ways to select the central node \(v\).
    
        \begin{subcases}
        
        \item Both arcs are oriented in the same direction with respect to \(v\) (both incoming or both outgoing).
            They cannot lie on a common directed cycle. There are two such orientation patterns. The probability of having two such arcs is \(p^2\). By \Cref{prop:EdgeCharSEP}, this contributes
            \[
            2 \cdot 3 \binom{n}{3} p^2 = 6 \binom{n}{3} p^2.
            \]
    
        \item The two arcs are oriented in opposite directions at \(v\) (one incoming, one outgoing).
            There are two such orientation patterns.
            These arcs form an edge in \(\SEP_{n,p}\) if they are not part of a directed 3-cycle or any directed 4-cycle.
            A 3-cycle is formed if the arc connecting the endpoints of \(e\) and \(f\) exists (probability \(p\)).
            A 4-cycle is formed via a fourth node \(w\) if the two arcs connecting \(w\) to the endpoints of \(e\) and \(f\) exist (probability \(p^2\)).
            Since the potential 3-cycle and the \((n-3)\) potential 4-cycles (one for each remaining node \(w\)) require disjoint sets of arcs, the events of these cycles occurring are independent.
            The probability that none of them occur is \((1 - p)(1 - p^2)^{n-3}\).
            This yields the contribution
            \[
            6 \binom{n}{3} p^2 (1 - p)(1 - p^2)^{n-3}.
            \]
            
        \end{subcases}
        
    \standarditem[Conclusion.]
        Summing these contributions yields the explicit formula.
        For the asymptotic behaviour, note that \(1-p^2 = (1-p)(1+p) \simeq 1-p\).
        The term from Case~1 is of order \(n^4 p^2 (1-p)\), while the contribution from Case~2 is of order \(n^3 p^2\).
        Comparing these magnitudes yields the claimed asymptotics.
        \qedhere
        
    \end{maincases}
    
\end{proof}

\begin{rem}\label{rem:SEPexpSubgraphs}
    \Cref{prop:EdgeCharSEP} reduces the edge counting in \(\SEP_{n,p}\) to counting pairs of arcs in \(G_{n,p}\) that cannot be completed to \(3\)- or \(4\)-cycles.
    Assume that \(\limsup p < 1\). Under this restriction, \Cref{expectation_sep} implies \(\EE[K_{n,p}] \simeq n^4 p^2\).
    Although different configurations contribute, the dominant term arises from pairs of disjoint arcs not contained in a \(4\)-cycle (\Cref{lem:exp-case:empty}).
    In this regime, the restriction of excluding pairs in a \(4\)-cycle only affects the scaling factor, not the asymptotic order.
    This matches the expected number of subgraphs in \(G_{n,p}\) consisting of two disjoint arcs:
    Let \(G_0\) be a graph formed by two disjoint arcs, and let \(S_{n,p}\) be the number of copies of \(G_0\) in \(G_{n,p}\). Then
    \[
    \EE[S_{n,p}] = \binom{n}{2} \binom{n-2}{2} \cdot \frac{1}{2} \cdot p^2 \simeq n^4 p^2,
    \]
    which matches the order of \(\EE[K_{n,p}]\). For an introduction to random subgraph counting, see e.g., \cite[Chapter~3]{JLR00}.
    However, this parallel with subgraph counting does not extend to the variance, as we will see in \Cref{rem:SEPvarSubgraphs}.
\end{rem}

With \Cref{expectation_sep} in hand, we have a precise count for \(\EE[K_{n,p}]\). In the sections that follow, we analyze the variance of \(K_{n,p}\) and establish the corresponding central limit theorem, thus finalizing the proof of \Cref{thm:Main} for edges of the symmetric edge polytope.

\subsection{Variance}

Having determined the expectation of \(K_{n,p}\), we now turn to its variance. Rather than providing an explicit formula for \(\var(K_{n,p})\), which is in principle possible, we focus on its asymptotic behavior and establish the asymptotic growth rate.

\begin{thm}\label{variance_sep}
    Let \(K_{n,p}\) denote the number of edges of the random symmetric edge polytope \(\SEP_{n,p}\).
    Then, as $n \to \infty$,
    \begin{align*}
        \var(K_{n,p})
        \;&\simeq\;
             n^6 p^3 (1-p) \bigl( 1 - \sqrt{2}p \bigr)^2
            + n^5 p^3 (1-p)
            + n^4 p^2 (1-p)
        \\
        \;&\simeq\;
            \begin{cases}
                 n^6 p^3 (1-p) \bigl( 1 - \sqrt{2}p \bigr)^2 + n^5 p^3 (1-p), & \text{if } p \gg n^{-2}, \\
                 n^4 p^2, & \text{if } p \lesssim n^{-2}.
            \end{cases}
    \end{align*}
    In particular, $\var(K_{n,p}) \to \infty$ if and only if \(p \gg n^{-2}\) and \(1-p \gg n^{-6}\).
\end{thm}

\begin{proof}

    To analyze \(\var(K_{n,p})\), we use the fact that \(K_{n,p}\) can be written as a sum of indicator random variables over pairs of (directed) arcs in the complete graph on \([n]\).
    Given an undirected arc \(e\), we use the symbol \(\arr{e}\) to denote a corresponding directed arc (without specifying its orientation). Specifically, for a pair of directed arcs \(\{\arr{e},\arr{f}\}\), let
    \(
        \indicator_{\{\arr{e},\arr{f}\}}
    \)
    be the indicator variable of the event that vertices corresponding to \(\arr{e}\) and \(\arr{f}\) exist and form an edge in \(\SEP_{n,p}\). This allows us to write
    \[
        K_{n,p}
        \;=\;
        \sum_{\{\arr{e},\arr{f}\}} \indicator_{\{\arr{e},\arr{f}\}},
    \]
    where the sum is over all unordered pairs \(\{\arr{e},\arr{f}\}\) of directed arcs in the complete graph.
    The variance then expands as
    \[
        \var(K_{n,p})
        \;=\;
        \sum_{\{\arr{e},\arr{f}\}} \sum_{\{\arr{g},\arr{h}\}}
        \mathrm{Cov}\bigl(\indicator_{\{\arr{e},\arr{f}\}},\,\indicator_{\{\arr{g},\arr{h}\}}\bigr).
    \]
    To evaluate the asymptotic growth of this double sum, we group the covariances according to how the arcs inside each of the two pairs of underlying undirected arcs \(\{e,f\}\) and \(\{g,h\}\) overlap.
    Each scenario yields different combinatorial factors, involving different powers of $n$ and different probabilities.
    Throughout this proof, we use the notation $(n)_k \coloneqq k! \binom{n}{k} \simeq n^k$ for $k \in \NN$ as $n \to \infty$. The cases are illustrated in \Cref{fig:variance}.
    
    \begin{figure}[t]
        \centering%
        \input{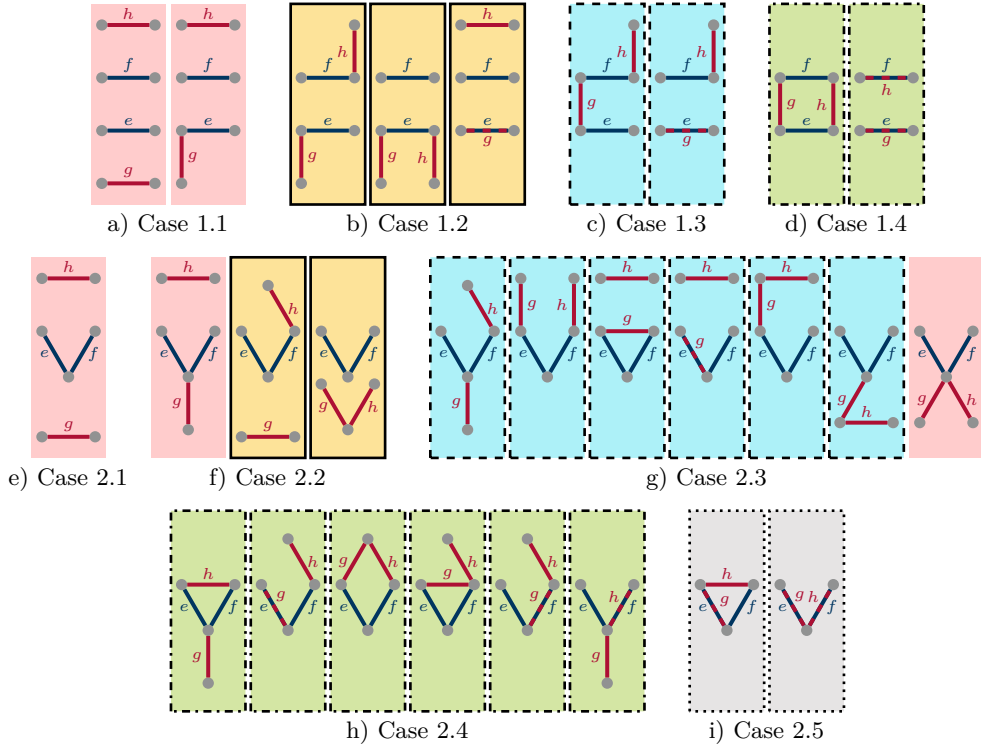}%
        \tikzset{every node/.style={font=\normalfont, inner sep=4pt}}%
        \caption{%
            The configurations arising in the proof of \Cref{variance_sep} (up to symmetry);
            the graphs are colored to indicate the order of $n$ in their contribution to the covariance:
            \raisebox{-5pt}{%
            \strut%
            \protect\caserateN6P0 \protect\tikz\protect\node[graphbox] {$n^6$};
            \protect\caserateN5P0 \protect\tikz\protect\node[graphbox] {$n^5$};
            \protect\caserateN4P0 \protect\tikz\protect\node[graphbox] {$n^4$};
            \protect\caserateN3P0 \protect\tikz\protect\node[graphbox] {$n^3$};
            \protect\caserateN0P0 \protect\tikz\protect\node[graphbox] {$0$ \rlap{\phantom{$n^0$}}\smash{(independent)}};%
            },
            where all configurations in Case 2 contain an additional exponentially decreasing factor.
        }%
        \label{fig:variance}%
    \end{figure}
    
    \begin{maincases}
    
    \item[\( e\cap f = \emptyset \) and \( g\cap h = \emptyset \)]\label{case:SEPvariance:1}
        Here each pair lies on four distinct nodes. As in the proof of \Cref{expectation_sep}, there is a unique directed 4-cycle containing a fixed pair of directed arcs. We distinguish further subcases based on the intersection size of the two sets of four nodes.
    
        \begin{subcases}
    
        \item[\( \vert (e \cup f) \cap (g \cup h) \vert \leq 1 \)]\label{subcase:SEPvariance:11}
            The pairs of arcs share at most one node. Their induced 4-cycles cannot share an arc. Consequently, \(\indicator_{\{\arr{e},\arr{f}\}}\) and \(\indicator_{\{\arr{g},\arr{h}\}}\) are independent random variables, and their covariance vanishes.
            
        \item[\( \vert (e \cup f) \cap (g \cup h) \vert = 2 \)]\label{subcase:SEPvariance:12}
            We assume the arcs share exactly two nodes. There are three possible configurations (up to symmetry), shown in \nameCref{fig:variance}~\hyperref[fig:variance]{2b)}.
            
            In the first configuration, there are two disjoint paths on three nodes, each consisting of exactly one arc from \(\{e,f\}\) and one from \(\{g,h\}\). By renaming if necessary, we assume \(e\) forms a 2-path with \(g\), and \(f\) with \(h\). Under this configuration, the respective directed 4-cycles can overlap only on the arc \(\ell\) between the nodes \(e\cap g\) and \(f\cap h\).
            We obtain such a configuration by choosing six nodes, dividing them into two equally-sized components, choosing the common node for each component, choosing which arcs are paired, assigning which pair is \(\{e,f\}\), and choosing one of the two possible orientations for each pair.
            This yields
            \[
                \binom{n}{6} \cdot \frac{\binom{6}{3}}{2} \cdot 3^2 \cdot 2 \cdot 2 \cdot 2^2
                = 2 (n)_6
            \]
            configurations. The arc \(\ell\) impacts both pairs, so these configurations yield a covariance of
            \[
                p^4 \Bigl(
                    \bigl( p(1-p)^2 + (1-p) \bigr)
                    - (1-p^2)^2
                \Bigr)
                = p^7(1-p),
            \]
            resulting in a contribution of \(2 (n)_6 p^7(1-p)\).
            
            The second configuration consists of one path on four nodes and one arc disjoint from the path. This can be obtained by choosing six nodes, choosing components of size four and two, choosing the middle arc of the 3-path, connecting the remaining nodes, and assigning which pair is \(\{e,f\}\).
            The indicator variables are dependent only if the 3-path is a subgraph of one of the respective 4-cycles.
            For the pair of arcs connected via the 3-path there are only two possible orientations, whereas for the other pair there are four possible orientations.
            This yields
            \[
                \binom{n}{6} \cdot \binom{6}{2} \cdot \binom{4}{2} \cdot 2 \cdot 2 \cdot 2 \cdot 4
                = 4 (n)_6
            \]
            configurations. Since the middle piece of the 3-path is already present to build the 4-cycle containing the other pair of arcs, they contribute a covariance of
            \[
                p^4 \Bigl(
                    \bigl( 1-p^2 \bigr)(1-p) - \bigl( 1-p^2 \bigr)^2
                \Bigr)
                = - p^5 (1-p) \bigl( 1-p^2 \bigr),
            \]
            resulting in a total contribution of \(-4 (n)_6 p^5(1-p)(1-p^2)\).

            In the third configuration, the two pairs share one arc and exactly two nodes.
            There are
            six nodes to be chosen,
            two of which have to be selected to form the shared arc.
            There are three ways to construct two arcs among the remaining four nodes,
            two possibilities to assign which pair of arcs is \(\{e,f\}\),
            and two orientations per arc.
            We obtain
            \[
                \binom{n}{6} \cdot
                \binom{6}{2} \cdot
                3 \cdot
                2 \cdot
                2^4
                = 2 (n)_6,
            \]
            such configurations.
            Their induced 4-cycles do not overlap except for the shared arc, which yields the covariance
            \[
                p^3 \bigl( 1-p^2 \bigr)^2 - p^4 \bigl( 1-p^2 \bigr)^2
                = p^3 (1-p) \bigl( 1-p^2 \bigr)^2.
            \]
            This leads to an overall contribution of \(2 (n)_6 p^3 (1-p) (1-p^2)^2\).

            Taking together all contributions from this subcase, we get a total of
            \begin{align*}
                 & 2 (n)_6 p^7 (1-p)
                 - 4 (n)_6 p^5 (1-p) \bigl( 1-p^2 \bigr)
                 + 2 (n)_6 p^3 (1-p) \bigl( 1-p^2 \bigr)^2
                \\
                &= 2 (n)_6 p^3 (1-p) \bigl( 1 - 2p^2 \bigr)^2 \\
                &\simeq n^6 p^3 (1-p) \bigl( 1 - \sqrt{2}p \bigr)^2.
            \end{align*}
            
        \item[\( \vert (e \cup f) \cap (g \cup h) \vert = 3 \)]\label{subcase:SEPvariance:13}
            There are two possible configurations (up to symmetry) in which the pairs of arcs share three nodes, shown in \nameCref{fig:variance}~\hyperref[fig:variance]{2c}.
            
            In the first configuration, $e,f,g,h$ form a path of length four with alternating arcs between pairs. There are \((n)_5/2\) ways to choose such a path and two ways to pick which pair of arcs is \(\{e,f\}\), yielding \((n)_5\) cases in total. The covariance depends on the orientation of each pair:
            \begin{itemize}
                \item There are four orientations where each 4-cycle includes one arc of the other pair. This results in a covariance of
                \( p^4 \bigl( (1-p)^2 - ( 1-p^2 )^2 \bigr)
                = -p^5 (1-p)^2 (2+p)\).
                \item There are eight orientations where only one 4-cycle contains an arc of the other pair. In these cases the other 4-cycle does not overlap with anything, yielding a covariance of
                \( p^4 \bigl( (1-p) ( 1-p^2 ) - ( 1-p^2 )^2 \bigr)
                = - p^5 (1-p) (1-p^2)\).
                \item Finally, there are four orientations where no cycle contains an arc of the other pair. In this case their 4-cycles share one arc, which yields a covariance of
                \( p^4 \bigl( p (1-p)^2 + (1-p) - ( 1-p^2 )^2 \bigr)
                = p^7 (1-p)\).
            \end{itemize}
            Summing these contributions gives in total
            \[
                -4 (n)_5 p^5 (1-p) \Bigl( 4 (1 - p^2) - p \Bigr).
            \]
            
            In the second configuration, the two pairs share one arc (including its two adjacent nodes) and additionally one more node.
            This configuration can be obtained by choosing five nodes, picking the shared arc, choosing the arc that does not appear among the remaining three nodes, and then choosing which arc of the 2-path belongs to which pair.
            This yields
            \[
                \binom{n}{5} \cdot \binom{5}{2} \cdot 3 \cdot 2
                = \frac{1}{2} (n)_5
            \]
            cases. Depending on the orientation there are two possible results for the covariance. Without loss of generality, we assume that $f$ and $h$ are the arcs forming the 2-path. In every possible orientation, both 4-cycles must contain one arc connecting the node given by \(f \cap h\) to the arc \(e=g\).
            \begin{itemize}
                \item If this arc is not the same for both 4-cycles, then the covariance only depends on the shared arc \(e=g\), which yields \( (p^3-p^4) ( 1-p^2 )^2 = p^3 (1-p) (1-p^2)^2\).
                \item If both 4-cycles overlap in this arc, then the cycles share two arcs in total, which yields a covariance of \( p^3 ( p(1-p)^2 + (1-p) ) - p^4 ( 1-p^2 )^2 = p^3 (1-p) (p^3 + (1-p^2)^2)\).
            \end{itemize}
            Note that each covariance arises in eight orientations. Putting both terms from this configuration together, we arrive at
            \[
                4 (n)_5 p^3 (1-p) \Bigl( p^3 + 2\bigl( 1-p^2 \bigr)^2 \Bigr).
            \]
            Collecting all contributions for this subcase yields
            \begin{align*}
                & -4 (n)_5 p^5 (1-p) \Bigl( 4 (1 - p^2)- p \Bigr)
                + 4 (n)_5 p^3 (1-p) \Bigl( p^3 + 2\bigl( 1-p^2 \bigr)^2 \Bigr)
                \\
                &= 8 (n)_5 p^3 (1-p) \Bigl( p^3 (1-p) + \bigl( 1-2p^2 \bigr)^2 \Bigr)
                \\
                &\simeq n^5 p^3 (1-p).
            \end{align*}
            
        \item[\( \vert (e \cup f) \cap (g \cup h) \vert = 4 \)]\label{subcase:SEPvariance:14}
            If \(e,f\) and \(g,h\) share four nodes, the pairs of arcs either form a 4-cycle or they are identical, i.e., \(\{e,f\}=\{g,h\}\).
            
            First, consider the case where the arcs form a 4-cycle. We can construct such a case by choosing four nodes, choosing two arcs between them to be \(e,f\), and finally choosing the other pair. This leads to
            \[
                \binom{n}{4} \cdot \frac{\binom{4}{2}}{2} \cdot 2
                = \frac{1}{4} (n)_4
            \]
            terms. The covariance depends on the orientation of the arcs:
            \begin{itemize}
                \item There are four ways to orient the arcs such that both induced 4-cycles match the one consisting of \(e,f,g,h\). In this case, the respective edges in the polytope \(\SEP_{n,p}\) can never be present simultaneously, resulting in a covariance of \(-\bigl( p^2 (1-p^2) \bigr)^2\).
                \item Eight different orientations lead to a configuration where only one of the induced 4-cycles matches the one formed by \(e,f,g,h\). Again, the respective edges can never be present simultaneously, yielding the same covariance of \(-\bigl( p^2 (1-p^2) \bigr)^2\) as in the previous case.
                \item Finally, there are four orientations such that none of the induced 4-cycles includes the other pair of arcs. In this case both 4-cycles are completed using the same arcs, yielding a covariance of \( p^4 \bigl( ( 1-p^2 ) - ( 1-p^2 )^2 \bigr) = p^6(1-p^2) \).
            \end{itemize}
            The total contribution from this configuration is
            \[
                - (n)_4 p^4 \bigl( 1-p^2 \bigr) \bigl( 3 - 4p^2 \bigr).
            \]
            
            Next, consider the case where the pairs of arcs are identical. This configuration yields
            \[
                \binom{n}{4} \cdot \frac{\binom{4}{2}}{2}
                = \frac{1}{8} (n)_4
            \]
            terms. We evaluate the different orientations:
            \begin{itemize}
                \item There are eight ways to orient the arcs such that \(\{\arr{e},\arr{f}\}=\{\arr{g},\arr{h}\}\) (or with flipped signs), so that their respective 4-cycles match. This yields a covariance of \(p^2(1-p^2) - p^4(1-p^2)^2\).
                \item There are eight ways to orient the arcs such that the pairs differ in sign or assignment so that their respective 4-cycles do not match. This yields a covariance of
                \(
                    p^2(1-p^2)^3
                \).
            \end{itemize}
            This leads to a total contribution of
            \[
                (n)_4 p^2 \bigl( 1-p^2 \bigr) \Bigl( 1 - p^2 (1-p^2) + (1-p^2)^2 \Bigr).
            \]
            Combining the cycle and identity cases yields terms of order \(n^4 p^2 (1-p)\).
    
        \end{subcases}
    
    \item[\(\vert e \cap f \vert = 1\) or \(\vert g \cap h \vert = 1\)]\label{case:SEPvariance:2}
        At least one of the pairs of arcs \(\{e,f\}\) or \(\{g,h\}\) shares a node and thus forms a 2-path. By renaming if necessary, we assume that \(\vert e \cap f \vert = 1\).
        We claim that all contributions from this case are of smaller order than those from \Cref{case:SEPvariance:1} and thus negligible.
        To prove this, we bound the absolute value of every individual covariance.
        Let \(k \in \{ 3,4,5,6,7 \} \) be the number of nodes and \(\ell \in \{ 2,3,4 \} \) be the number of arcs involved in the configuration.
        For \(k=7\) the covariance is zero.
        Let $k \leq 6$.
        For an edge indicator to be 1, both respective arcs must be present. Further, for \(\{e,f\}\), they must not form a 4-cycle with any of the \(n-k\) other nodes.
        Thus,
        \begin{align*}
            \big\vert \mathrm{Cov}\bigl(\indicator_{\{\arr{e},\arr{f}\}}, \indicator_{\{\arr{g},\arr{h}\}}\bigr) \big\vert
            &\leq \max \Bigl\{
                    \EE\bigl[ \indicator_{\{\arr{e},\arr{f}\}} \indicator_{\{\arr{g},\arr{h}\}} \bigr]
                    \, , \,
                    \EE\bigl[ \indicator_{\{\arr{e},\arr{f}\}} \bigr] \EE\bigl[ \indicator_{\{\arr{g},\arr{h}\}} \bigr]
                \Bigr\}
            \\
            &\leq p^\ell \bigl( 1-p^2 \bigr)^{n-k}
            \\
            &\leq p^\ell e^{-(n-k)p^2}.
        \end{align*}
        For every configuration, the number of individual covariances of this type is of order at most \(n^k\), resulting in a total contribution of order at most \(n^k p^\ell e^{-np^2}\).
        For $p \gg n^{-\frac12}$ the term \(e^{-(n-k)p^2}\) ensures that these contributions are negligible compared to the polynomial terms derived in \Cref{case:SEPvariance:1}.
        If $p \lesssim n^{-\frac12}$, there is no exponential decay and the contributions are at most of order $n^6p^4$, $n^5p^3$, $n^4p^3$, or $n^3p^2$.
        In comparison, for $p \to 0$ the results from \Cref{case:SEPvariance:1} are of order $n^6p^3$ or $n^4p^2$ and therefore dominant.
    
    \standarditem[Conclusion.]
        Collecting all dominant terms, we obtain contributions of order \(n^6 p^3 (1-p) (1 - \sqrt{2}p)^2\), \(n^5 p^3 (1-p)\), and \(n^4 p^2 (1-p)\). All contributions from \Cref{case:SEPvariance:2} are of smaller order. This completes the proof.
        \qedhere
    
    \end{maincases}

\end{proof}

\begin{rem}
    Under the assumptions that $\limsup_{n\to \infty} p<1$ and \(p \gg n^{-2}\), \Cref{variance_sep} implies
    \[
        \var(K_{n,p})
        \;\simeq\; n^6 p^3 \bigl( 1 - \sqrt{2}p \bigr)^2 + n^5 p^3
        \;\to\; \infty.
    \]
    If, additionally, $\frac{1}{\sqrt{2}}$ is not an accumulation point of $p$, the leading term dominates, and we recover the generic growth rate
    \[
        \var(K_{n,p})
        \;\simeq\; n^6 p^3.
    \]
    However, if $p$ approaches the critical value $\frac{1}{\sqrt{2}}$, the leading coefficient vanishes. In this regime, the variance grows more slowly at a rate strictly less than $n^6$, but at least of order $n^5$.
    While the algebraic cancellation leading to this suppression is explicit in our covariance analysis (specifically between Subcases 1.2 and 1.3), a heuristic explanation for why the geometric fluctuations are minimized at this specific density remains an open question.
\end{rem}

\begin{rem}\label{rem:SEPvarSubgraphs}
    In \Cref{rem:SEPexpSubgraphs}, we highlighted parallels between edge counting in $\SEP_{n,p}$ and subgraph counting in $G_{n,p}$.
    This parallel breaks down at the level of the variance.
    Let \(G_0\) be the graph consisting of two disjoint arcs, and let \(S_{n,p}\) be the number of copies of $G_0$ in \(G_{n,p}\).
    Standard results on subgraph counts (see, e.g., \cite[Lemma~3.5]{JLR00}) imply that for \( p \gg n^{-2} \) and $\limsup_{n\to\infty} p < 1$, the variance satisfies
    \[
        \var(S_{n,p}) \;\simeq\; \frac{(n^4p^2)^2}{ n^2p } \;=\; n^6p^3.
    \]
    This matches the generic order of $\var(K_{n,p})$ derived in \Cref{variance_sep}.
    However, the variance of the subgraph count \(S_{n,p}\) does not exhibit any degeneracy at $p = \frac{1}{\sqrt{2}}$.
    The suppression of the leading variance term observed for the symmetric edge polytope appears to be a phenomenon unique to the geometric construction, arising from the specific interplay of forbidden $3$- and $4$-cycles, which has no direct analogue in classical subgraph counting.
\end{rem}

\subsection{Discrete gradients}

The second step in establishing the central limit theorem for the number of edges \( K_{n,p} \) of the random symmetric edge polytope \( \SEP_{n,p} \) involves controlling the discrete gradients of \( K_{n,p} \) and estimating their moments. To this end, we introduce the following notation. Let \( G = (V, E) \) be a graph with set of nodes  \(V=[n]\). For notational convenience, we write \( st \in \binom{V}{2} \) instead of \( \{s,t\} \in \binom{V}{2} \). We define an \emph{\( s \)--\( t \) path of length \( i \geq 1 \)} in \( G \) as a sequence of distinct nodes
\[
s = v_0, v_1, v_2, \dots, v_i = t
\]
such that \( \{v_\ell, v_{\ell+1}\} \in E \) for all \( \ell = 0, 1, \dots, i-1 \). The length \( i \) refers to the number of arcs traversed.
When the endpoints \(s\) and \( t \) are not specified, we simply speak of \(i\)-paths.
We write
\(G + e\) for the graph obtained from \(G\)
by adding \(e=st \in \binom{V}{2}\), if not already present,
that is, \(G + e \coloneqq (V,E \cup \{e\}) \),
and \(G - e\) for the graph obtained
by removing \(e\), if present,
that is, \(G - e \coloneqq (V,E \setminus \{e\}) \).
We denote \( q \coloneqq 1 - p \).
\begin{lemma}\label{Dfbound_SEP}
    Let \( G_{n,p} \) be an Erd\H{o}s--R\'enyi random graph and let \( E_{n,p} \) denote the number of its arcs. Fix \( e = st \in \binom{[n]}{2} \), and let \( W_{n,p}(i) \) denote the number of \( s \)--\( t \) paths of length \( i \geq 1 \) in \( G_{n,p} \). Then the following bound holds almost surely:
    \[
    \left| \D_{e} K_{n,p} \right| \leq \sqrt{pq} \bigl( 4E_{n,p} + 2W_{n,p}(2) + 6W_{n,p}(3) \bigr).
    \]
\end{lemma}
\begin{proof}
Recall from \Cref{sec:backgroundprob}
that \((K_{n,p})^+_e\) and \((K_{n,p})^-_e\)
denote the numbers of edges of the symmetric edge polytope
corresponding to \(G_{n,p}+e\) and \(G_{n,p}+e\), respectively.
Then
\[
\D_e K_{n,p}=\sqrt{pq}\,\bigl( (K_{n,p})^+_e - (K_{n,p})^-_e \bigr).
\]
We bound \( \bigl| (K_{n,p})^+_e - (K_{n,p})^-_e \bigr| \) by counting edges of the corresponding polytope that are created and destroyed when passing from \(G_{n,p}-e\) to \(G_{n,p}+e\).

\smallskip\noindent
\emph{Created edges.}
Adding \(e\) introduces exactly two new vertices of the polytope, namely \(\ee_{s,t}\) and \(\ee_{t,s}\), corresponding to the directed arcs \(s\to t\) and \(t\to s\). Any newly created polytope edge must be incident to one of these two vertices.
The polytope associated with \(G_{n,p}-e\) has \( 2 (E_{n,p})^-_e\) vertices, where \((E_{n,p})^-_e\) denotes the number of arcs of \(G_{n,p}-e\).
Hence each of the two new vertices can form an edge with at most \((E_{n,p})^-_e\) existing vertices, so the number of created edges is at most
\[
2 (E_{n,p})^-_e + 2 (E_{n,p})^-_e
= 4 (E_{n,p})^-_e
\le 4E_{n,p}.
\]

\smallskip\noindent
\emph{Destroyed edges.}
By \Cref{prop:EdgeCharSEP}, an edge between two (non-antipodal) vertices of \( \SEP_{n,p} \) is present if and only if there is no directed \(3\)- or \(4\)-cycle in \(G_{n,p}\) containing both corresponding directed arcs. Therefore, any polytope edge that is present for \(G_{n,p}-e\) but absent for \(G_{n,p}+e\) must be destroyed by a directed \(3\)- or \(4\)-cycle that appears only after adding \(e\). Such a new directed cycle must use one of the new directed arcs \(s\to t\) or \(t\to s\). Consequently, destroyed edges can be charged to directed cycles obtained by concatenating \(s\to t\) or \(t\to s\) with an \(s\)--\(t\) path already present in \(G_{n,p}-e\).

If there is an \(s\)--\(t\) path of length \(2\), then together with \(s\to t\) (resp.\ \(t\to s\)) it forms a directed \(3\)-cycle. Each such directed \(3\)-cycle contains exactly two arcs that are already present in \(G_{n,p}-e\), and hence can destroy at most one pre-existing polytope edge (namely the edge between the two vertices corresponding to these two old arcs). Since there are two choices of orientation (using \(s\to t\) or \(t\to s\)), the total number of destroyed edges arising from length-\(2\) paths is at most \(2W_{n,p}(2)\).

Similarly, each \(s\)--\(t\) path of length \(3\) yields, together with \(s\to t\) (resp.\ \(t\to s\)), a directed \(4\)-cycle. Such a directed \(4\)-cycle contains three arcs that are already present in \(G_{n,p}-e\), and hence can destroy at most \(\binom{3}{2}=3\) pre-existing polytope edges among the corresponding three vertices. Again accounting for both orientations, the total number of destroyed edges arising from length-\(3\) paths is at most \(6W_{n,p}(3)\).

\smallskip
Combining the two parts and using that\(\bigl| (K_{n,p})^+_e - (K_{n,p})^-_e \bigr|\) is bounded by the sum of the number of created and destroyed edges, we obtain
\[
\bigl| (K_{n,p})^+_e - (K_{n,p})^-_e \bigr|
\le 4E_{n,p} + 2W_{n,p}(2) + 6W_{n,p}(3).
\]
Multiplying by \(\sqrt{pq}\) yields the asserted bound.
\end{proof}

In view of a future application of \Cref{prop:2ndOrderPoincare}, we need to bound the fourth moments of the random variables $E_{n,p}$, $W_{n,p}(2)$ and $W_{n,p}(3)$.

\begin{lemma}\label{l:W123_bound_sep}
    Let \( E_{n,p} \) be the number of arcs in \( G_{n,p} \), and let \( W_{n,p}(2) \) and \( W_{n,p}(3) \) denote the number of \( s \)--\( t \) paths of length 2 and 3, respectively, for a fixed pair \( e=st \in \binom{[n]}{2} \).
    Then, as \(n \to \infty\), the following bounds hold:
    \begin{align*}
        \EE[E_{n,p}^4] &\simeq
        \begin{cases}
            n^8 p^4, & \text{if } p \gg n^{-2}, \\
            n^2 p,   & \text{if } p \lesssim n^{-2},
        \end{cases} \\
        \EE[W_{n,p}(2)^4] &\simeq
        \begin{cases}
            n^4 p^8, & \text{if } p \gg n^{-1/2}, \\
            n p^2,   & \text{if } p \lesssim n^{-1/2},
        \end{cases} \\
        \EE[W_{n,p}(3)^4] &\simeq
        \begin{cases}
            n^8 p^{12}, & \text{if } p \gg n^{-2/3}, \\
            n^2 p^3,   & \text{if } p \lesssim n^{-2/3}.
        \end{cases}
    \end{align*}
    In particular, using  \Cref{Dfbound_SEP}, we obtain
    \[
        \EE[\left(\D_{e} K_{n,p}\right)^4]
        \lesssim
        p^2q^2 \cdot
        \begin{cases}
            n^8 p^4, & \text{if } p \gg n^{-2}, \\
            n^2 p,   & \text{if } p \lesssim n^{-2}.
        \end{cases}
    \]
\end{lemma}

\begin{proof}
    We begin with the fourth moment of \( E_{n,p} \), the number of arcs in \( G_{n,p} \). Since \( E_{n,p} \) follows a binomial distribution \( \mathrm{Bin} \left( \binom{n}{2}, p \right) \), we use the standard formula for the fourth moment of a binomial variable \( X \sim \mathrm{Bin}(m, p) \),
    \[
        \EE[X^4] = m^4 p^4 + 6 m^3 p^3 + 7 m^2 p^2 + m p.
    \]
    This immediately yields
    \[
        \EE[E_{n,p}^4] \simeq n^8 p^4 + n^2 p,
    \]
    which implies the stated cases.
    
    Next, consider \( W_{n,p}(2) \), the number of \( s \)--\( t \) paths of length two for a fixed pair \( e = st \). Each such path corresponds to a node \( r \in [n] \setminus \{s,t\} \) such that both arcs \( \{s,r\} \) and \( \{r,t\} \) are present in the graph. Since these paths are edge-disjoint, the corresponding indicator variables are independent. Therefore, \( W_{n,p}(2) \sim \mathrm{Bin}(n - 2, p^2) \).
    Applying the fourth moment formula again yields
    \[
        \EE[W_{n,p}(2)^4] \simeq (np^2)^4 + np^2,
    \]
    recovering the claimed bounds.
    
    We now turn to \( W_{n,p}(3) \). Each 3-path corresponds to a sequence \( s, v_1, v_2, t \) where the arcs \( \{s,v_1\}, \{v_1,v_2\}, \{v_2,t\} \) are present. The expected number of such paths is of order \( n^2 p^3 \).
    To estimate the fourth moment, we consider 4-tuples of such paths. The dominant contribution arises either from 4-tuples of disjoint paths or from a single path taken to the fourth power.
    \begin{itemize}
        \item If the four paths are edge-disjoint (sharing no nodes other than \(s,t\)), there are of order \( n^8 \) choices, and the probability is \( (p^3)^4 = p^{12} \). This yields a term of order \( n^8 p^{12} \).
        \item If the paths are identical (maximal overlap), the contribution is of order \( n^2 p^3 \).
    \end{itemize}
    Other overlap types (sharing some nodes or edges) yield terms of intermediate order.
    Comparing \( n^2 p^3 \) with \( n^8p^{12}=(n^2 p^3)^4 \),
    we see that if \( n^2 p^3 \gg 1 \) (equivalently \( p \gg n^{-2/3} \)), the disjoint term dominates. Otherwise, the single path term dominates. Thus,
    \[
        \EE[W_{n,p}(3)^4] \simeq n^8 p^{12} + n^2 p^3.
    \]
    
    Finally, we use the bound from \Cref{Dfbound_SEP},
    \[
        \left| \D_e K_{n,p} \right| \leq \sqrt{pq} \bigl( 4E_{n,p} + 2W_{n,p}(2) + 6W_{n,p}(3) \bigr)
    \]
    together with the elementary inequality \( (a+b+c)^4 \leq 3^3 (a^4 + b^4 + c^4) \) for \( a,b,c \geq 0 \).
    Since \( E_{n,p} \) dominates both \( W_{n,p}(2) \) and \( W_{n,p}(3) \) in the relevant regimes (or is of the same order when \(p\) is large), the total fourth moment is bounded by the contribution of \( E_{n,p} \).
    Specifically, we have \( \EE[E_{n,p}^4] \simeq n^8 p^4 + n^2 p \). Multiplying by the prefactor \( (\sqrt{pq})^4 = p^2 q^2 \), we obtain
    \[
        \EE[\left( \D_e K_{n,p} \right)^4] \lesssim p^2q^2 (n^8 p^4 + n^2 p),
    \]
    which simplifies to the target rate. This completes the proof.
\end{proof}

Now that since we have bounded the fourth moment of the first-order discrete gradients, the only thing left to apply 
\Cref{prop:2ndOrderPoincare} are the second-order discrete gradients:

\begin{lemma}\label{DDbounds_SEP}
    Let $K_{n,p}$ be as before. Let $e,f\in \binom{[n]}{2}$ with $e\neq f$. Then,
    for $e\cap f = \emptyset$ it holds
    \begin{align*}
        \E\left[(\D_e\D_fK_{n,p})^4\right]
        &\lesssim p^4q^4,
    \intertext{and for $\vert e\cap f\vert= 1$ we have the bounds}
        \E\left[(\D_e\D_fK_{n,p})^4\right]
        &\lesssim
        p^4q^4 + n^4p^{12}q^4
        \simeq
        \begin{cases}
            n^4p^{12}q^4, &\text{if } p \gg n^{-\frac12},\\
            p^4q^4, &\text{if } p \lesssim n^{-\frac12}.
        \end{cases}
    \end{align*}
\end{lemma}
\begin{proof}\label{p:DDbounds_SEP}

    To ease notation and avoid double indices write $K\coloneqq K_{n,p}$ and $G \coloneqq G_{n,p}$.
    Let $\Bar{G}= G-e-f$ be the graph $G$ without the arcs $e,f$.
    Note that $\Bar{G}$ corresponds to the random variable $K_{e,f}^{-,-} \coloneqq ( K_e^- )_f^-$.
    
    Our goal is to express $K_{e,f}^{+,+}, K_{e,f}^{+,-}$ and $ K_{e,f}^{-,+} $ in terms of $K_{e,f}^{-,-}$.
    Our strategy is to analyze how the addition of $e$ and $f$ to $\Bar{G}$ can influence the number of edges of $\SEP_G$ and to then bound each of the occurring cases separately.
    To precisely name those cases, we define a random variable for each case.
    The relations among all random variables that will be defined in the following
    are visualized in \Cref{fig:DDvarNames}.
    \begin{figure}[htbp]
            \tikzset{
        SEPbubble/.style={circle,fill=RUBgray!70!white,minimum width=2cm},
        typeR/.style={Stealth-,draw,line width=1.5pt,color=OSNAred},
        typeS/.style={-Stealth,draw,line width=1.5pt,color=RUBgreen},
        typeT/.style={Stealth-,draw,line width=1.5pt,color=RUBblue},
        arrowlabel/.style={inner sep=0.2pt,auto=right},
    	every picture/.style={xscale=4,yscale=2.2},
    }%
    \centering
    \begin{tikzpicture}
        \path ( 0,0) node[name=G  ,SEPbubble] {$\SEP_{\Bar{G}}$};
        \path (-1,1) node[name=Ge ,SEPbubble] {$\SEP_{\Bar{G}+e}$};
        \path ( 1,1) node[name=Gf ,SEPbubble] {$\SEP_{\Bar{G}+f}$};
        \path ( 0,2) node[name=Gef,SEPbubble] {$\SEP_{\Bar{G}+e+f}$};
        \path[typeS,bend right= 25] (G)   edge node[arrowlabel]    {$S_e$}       (Ge);
        \path[typeR,bend right= 25] (Ge)  edge node[arrowlabel,']  {$R_{e,f}$} (G) ;
        \path[typeR,bend right= 40] (Ge)  edge node[arrowlabel]    {$R_e^{(f)}$}   (G) ;
        \path[typeS, bend left= 25] (G)   edge node[arrowlabel,']  {$S_f$}       (Gf);
        \path[typeR, bend left= 25] (Gf)  edge node[arrowlabel]    {$R_{e,f}$} (G) ;
        \path[typeR, bend left= 40] (Gf)  edge node[arrowlabel,']  {$R_f^{(e)}$}   (G) ;
        \path[typeT,bend right= 25] (Gef) edge node[arrowlabel,']  {$T_{e,f}$} (Ge) ;
        \path[typeT,bend right= 40] (Gef) edge node[arrowlabel]    {$T_e^{(f)}$}   (Ge) ;
        \path[typeT, bend left= 25] (Gef) edge node[arrowlabel]    {$T_{e,f}$} (Gf) ;
        \path[typeT, bend left= 40] (Gef) edge node[arrowlabel,']  {$T_f^{(e)}$}   (Gf) ;
        \path[typeS, bend left=  0] (G)   edge node[arrowlabel]    {$S_{e,f}$}   (Gef);
    \end{tikzpicture}
        \caption{Relevant variables in the proof of \Cref{DDbounds_SEP}}
        \label{fig:DDvarNames}
    \end{figure}
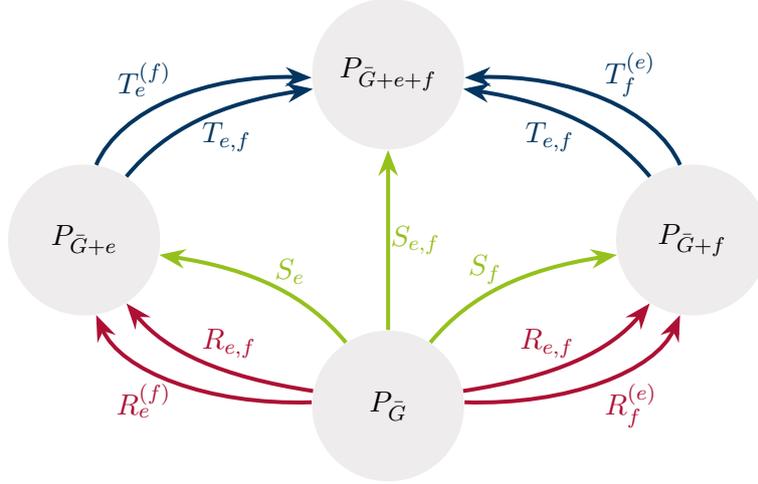
    
    Let $\{i,j\} = \{e,f\}$.
    By adding the arc $i$ to $\Bar{G}$,
    some new edges in $\SEP_{\Bar{G}}$ can appear and some previously present edges can disappear.
    We denote the number of new edges by $S_i$, i.e.,
    the number of edges of $\SEP_{\Bar{G}+i}$ that are not edges of $\SEP_{\Bar{G}}$.
    The number of disappearing edges is further split:
    $R_i^{(j)}$ denotes the number of edges that disappear
    if $i$ is added to $\Bar{G}$ but not if $j$ is added instead.
    $R_{e,f}$ denotes the number of edges that disappear
    if one of the arcs $e$ or $f$ is added to $\Bar{G}$, no matter which. 
    
    Some edges may only appear or disappear if both arcs, $e$ and $f$, are added to $\Bar{G}$,
    but not if only one of the arcs is added.
    Let $S_{e,f}$ denote the number of edges that appear in this case, i.e.,
    the number of edges of $\SEP_{\Bar{G}+e+f}$ that are neither in $\SEP_{\Bar{G}+e}$, $\SEP_{\Bar{G}+f}$, nor $\SEP_{\Bar{G}}$.
    Conversely,
    let $T_{e,f}$ be the number of edges that are in $\SEP_{\Bar{G}+e}$ and $\SEP_{\Bar{G}+f}$ (and thus also in $\SEP_{\Bar{G}}$), but not in $\SEP_{\Bar{G}+e+f}$.
    
    Finally, if $\Bar{G}+i$ is completed to $\Bar{G}+i+j = \Bar{G}+e+f$ by adding arc $j$,
    again some edges in the corresponding symmetric edge polytope can appear or disappear.
    Those disappearing are mostly already covered by $R_j^{(i)}$.
    We only have to take into account that those edges that previously appeared by adding $i$ (counted in $S_i$) can now disappear again by adding $j$.
    The number of these edges is denoted by $T_i^{(j)}$,
    i.e., the number of edges of $\SEP_{\Bar{G}+i}$ that are not edges of $\SEP_{\Bar{G}+i+j} = \SEP_{\Bar{G}+e+f}$ nor $\SEP_{\Bar{G}}$.
    On the other hand, by adding $j$ to $\Bar{G}+i$ some new edges may appear.
    The number of these edges is already covered by $S_j - T_j^{(i)}$.
    
    Bringing all these changes together,
    an application of the inclusion--exclusion principle leaves us with
    \begin{align*}
        K_{e,f}^{+,+} &= K_{e,f}^{-,-}
                            + S_e + S_f + S_{e,f}
                            - \bigl( R_e^{(f)} + R_f^{(e)} + R_{e,f} \bigr)
                            - \bigl( T_e^{(f)} + T_f^{(e)} + T_{e,f} \bigr),\\
        K_{e,f}^{+,-} &= K_{e,f}^{-,-} + S_e - \bigl( R_e^{(f)} + R_{e,f} \bigr),\\
        K_{e,f}^{-,+} &= K_{e,f}^{-,-} + S_f - \bigl( R_f^{(e)} + R_{e,f} \bigr).
    \end{align*}
    Thus,
    \begin{align}
        \D_e\D_f K
        &= pq \bigl( K_{e,f}^{+,+} - K_{e,f}^{+,-} - K_{e,f}^{-,+} + K_{e,f}^{-,-} \bigr) \nonumber \\
        &= pq \bigl( S_{e,f} - R_{e,f} - T_e^{(f)} - T_f^{(e)} - T_{e,f} \bigr). \label{proof_DDSEP_Konstanten}
    \end{align}
    We now need to bound each of these random variables.
    If necessary, we will distinguish cases based on whether $e$ and $f$ intersect.
    
    \begin{boundingcases}
    
    \item[$S_{e,f}$]\label{bound:Sef}
        This random variable denotes the number of those edges that only appear if $e$ and $f$ are both added to $\Bar{G}$.
        In particular, the relevant edges of $\SEP_{\Bar{G}+e+f}$ are those between vertices of the antipodal pairs corresponding to $e$ and $f$.
        Since $e$ and $f$ each give rise to exactly two vertices of the polytope, the almost sure bound $S_{e,f}\leq 4$ follows immediately.
        
    \item[$R_{e,f}$]\label{bound:Ref}
        $R_{e,f}$ is the number of edges of $\SEP_{\Bar{G}}$ that are neither in $\SEP_{\Bar{G}+e}$ nor $\SEP_{\Bar{G}+f}$.
        These correspond to pairs of directed arcs $\arr{m_1}, \arr{m_2}$ that do not lie on a $3$- or $4$-cycle in $\Bar{G}$ but do so in $\Bar{G}+e$ and $\Bar{G}+f$.
        
        In particular, for any such pair $\arr{m_1}, \arr{m_2}$ and $i \in \{e,f\}$ there has to exist a unique $3$- or $4$-cycle $C_i$ in $\bar{G}+i$ containing $m_1$, $m_2$, and $i$,
        and these two cycles $C_e$ and $C_f$ are different. 
        Note that this is not possible if $m_1 \cap m_2=\emptyset$ since in this case $\arr{m_1}$ and $\arr{m_2}$ induce a unique directed $4$-cycle (and no $3$-cycle), so that $C_e$ and $C_f$ cannot be different.
        So, we can assume $\vert m_1\cap m_2\vert =1$.
        Let $w\in m_1\cap m_2$ be this unique node,
        and let the other nodes in $m_1$ and $m_2$ be labeled $x_1$ and $x_2$
        such that $m_1=x_1w$, $m_2=x_2w$.
        
        If there was another pair of arcs $m'_1, m'_2$ forming a $2$-path with the same endpoints as $m_1,m_2$,
        i.e., $m_1'=x_1w', m_2'=x_2w'$ with $w' \neq w$,
        then $m_1,m_2,m'_2,m'_1$ would form a $4$-cycle on the nodes $x_1,w,x_2,w'$
        and would not have given rise to any edges of $\SEP_{\bar{G}}$
        (see \Cref{fig:D_eD_f1} for a picture of the configuration).
        Hence, for any given pair of endpoints $x_1,x_2$ there is at most one pair of arcs $m_1,m_2$
        contributing to $R_{e,f}$ by $2$ (one per orientation).
        
        \begin{boundingsubcases}
        
        \item[$e\cap f = \emptyset$]
            Since $e$ and $f$ both have to share at least one node with $m_1$ or $m_2$ and since $e\cap f=\emptyset$ by assumption, $C_e$ and $C_f$ have to be $4$-cycles.
            Moreover, as each node in a cycle has to have degree $2$ and $w\in m_1\cap m_2$, it follows that $w \notin e\cup f$ (see \Cref{fig:D_eD_f2}).
            Therefore, one of the nodes $x_1,x_2$ (which are the endpoints of the $2$-path consisting of $m_1,m_2$) must be in $e$, whereas the other node has to be in $f$.
            For any given pair of endpoints there can only be at most one pair of arcs relevant for $R_{e,f}$ as two pairs of arcs with the same endpoints would form a $4$-cycle.
            Since there are $2^2$ combinations of nodes of $e$ and $f$,
            further taking into account that there are $2$ different orientations of the $2$-path connecting these nodes,
            this yields $R_{e,f} \leq 8$ almost surely.
            
            
        \item[$e\cap f=\{v\}$ for some $v\in V$]
            Let the other nodes in $e$ and $f$ be labeled $u_e$ and $u_f$
            such that $e=u_e v$ and $f=u_f v$.
            
            First, assume that $C_i$ is a $3$-cycle for an $i \in \{e,f\}$.
            This implies that $m_1,m_2$ form a $2$-path with endpoints $u_i$ and $v$.
            There can be at most one $2$-path per $i \in \{e,f\}$ with these properties, as two such paths would form a $4$-cycle.
            Taking into account two possible orientations per $2$-path, there are at most four edges of $\SEP_{\bar{G}}$ that are destroyed by adding $e$ or $f$ due to their corresponding arcs lying on a $3$-cycle with $e$ or $f$.
            
            Now, assume that $C_e$ and $C_f$ are $4$-cycles.
            Lying on a $4$-cycle with $e$ and $f$ implies that the $2$-path consisting of $m_1=x_1w$ and $m_2=x_2w$ has to share exactly one of its endpoints $x_1,x_2$ with each of $e=u_ev$ and with $f=u_fv$.
            Sharing $u_e$ and $u_f$, i.e., $\{ x_1,x_2 \}=\{ u_e,u_f \}$ would correspond to the cycles $C_e$ and $C_f$ being identical.
            Such pairs are counted for $T_{e,f}$ but not for $R_{e,f}$. 
            
            So, we can assume that $v$ is among the shared nodes, $v \in \{ x_1,x_2 \}$.
            Without loss of generality, assume that $v = x_1$,
            such that
            \(m_1=vw,m_2=x_2w\).
            By the same argument as above, there cannot exist another pair of arcs $m'_1$, $m'_2$ with \(m_1'=vw', m_2'=x_2w'\) and $w' \neq w$,
            otherwise the arcs $m_1,m_2,m'_2,m'_1$ would form a $4$-cycle (see \Cref{fig:D_eD_f3})
            which contradicts the fact that $\arr{m_1}$ and $\arr{m_2}$ (with the appropriate orientation) form an edge of $\SEP_{\Bar{G}}$.
            This means we can count at most two removed edges -- considering directions -- for each $x_2 \in [n]$. 
            In order for $x_2$ to correspond to such a pair, the arcs $u_e x_2$ and $u_f x_2$ have to be contained in $\bar{G}$,
            as shown in \Cref{fig:D_eD_f4}.
            These correspond to $u_e$-$u_f$ paths of length $2$.
            
            So, we can see that 
            \[
                R_{e,f} \leq 2 W_{n,p}(2) + 4
            \]
            holds almost surely, where \( W_{n,p}(2) \) is the number of \( u_e \)--\( u_f \) paths of length \( 2 \) in \( G \),
            as known from \Crefrange{Dfbound_SEP}{l:W123_bound_sep}.
            
            \begin{figure}[tbp]%
                \begin{subfigure}{0.45\textwidth}%
    \centering%
    \begin{tikzpicture}[line width=1.5pt, line cap=round, line join=round]
        \draw[OSNAred] (-1,1) -- (0,2) node[midway, anchor=-30, inner sep=1pt] {$m_2$};
        \draw[OSNAred] (-1,1) -- (0,0) node[midway, anchor=30, inner sep=1pt] {$m_1$};
        \node[OSNAred,left] at (-1,1) {$w$};
        \draw[RUBgreen] (1,1) -- (0,2) node[midway, anchor=210, inner sep=1pt] {$m^\prime_2$};
        \draw[RUBgreen] (1,1) -- (0,0) node[midway, anchor=150, inner sep=1pt] {$m^\prime_1$};
        \node[RUBgreen,right] at (1,1) {$w^\prime$};
        \node[OSNAred,below] at (0,0) {$x_1$};
        \node[OSNAred,above] at (0,2) {$x_2$};
        \draw[RUBblue] (0,2) -- (2,2) node[midway, above] {$e$};
        \draw[RUBblue] (0,0) -- (2,0) node[midway, below] {$f$};
    \end{tikzpicture}%
    \caption{%
        A configuration where
        {\color{RUBblue}$e$} and {\color{RUBblue}$f$}
        lie in two different $4$-cycles each with edges
        {\color{OSNAred}$m_1$} and {\color{OSNAred}$m_2$},
        respectively
        {\color{RUBgreen}$m^\prime_1$} and {\color{RUBgreen}$m^\prime_2$}.%
        \\%
    }%
    \label{fig:D_eD_f1}%
\end{subfigure}%
\hfill%
\begin{subfigure}{0.45\textwidth}%
    \centering%
    \begin{tikzpicture}[line width=1.5pt, line cap=round, line join=round]
        \draw[OSNAred, -Stealth] (-1,1) -- (0,2) node[midway, anchor=-30, inner sep=1pt] {$\arr{m_2}$};
        \draw[OSNAred, Stealth-] (-1,1) -- (0,0) node[midway, anchor=30, inner sep=1pt] {$\arr{m_1}$};
        \node[OSNAred,left] at (-1,1) {$w$};
        \node[OSNAred,below] at (0,0) {$x_1$};
        \node[OSNAred,above] at (0,2) {$x_2$};
        \draw[OSNAgold,dotted] (0,0)--(2,2);
        \draw[OSNAgold,dotted] (2,0)--(0,2);
        \draw[RUBblue] (0,2) -- (2,2) node[midway, above] {$e$};
        \draw[RUBblue] (0,0) -- (2,0) node[midway, below] {$f$};
    \end{tikzpicture}%
    \caption{%
        The configuration in Case 1 
        for bounding $R_{e,f}$.%
        \\%
        \mbox{}\\%
    }%
    \label{fig:D_eD_f2}%
\end{subfigure}%
\linebreak%
\begin{subfigure}{0.45\textwidth}%
    \centering%
    \begin{tikzpicture}[line width=1.5pt, line cap=round, line join=round]
        \draw[OSNAred] (0,0) -- (1.5,1)  node[midway, anchor=150, inner sep=1pt] {$m_1$};
        \draw[OSNAred] (2,2) -- (1.5,1) node[midway, anchor=150, inner sep=1pt] {$m_2$};
        \node[OSNAred,anchor=150] at (1.5,1) {$w$};
        \draw[RUBgreen] (0,0) -- (1,1.5) node[midway, anchor=-70] {$m^\prime_1$};
        \draw[RUBgreen] (2,2) -- (1,1.5) node[midway, anchor=-60, inner sep=1pt] {$m^\prime_2$};
        \node[RUBgreen,anchor=-45] at (1,1.5) {$w^\prime$};
        \node[RUBblue,anchor=45] at (0,0) {$v = x_1$};
        \node[OSNAgold,right] at (2,2) {$x_2$};
        \node[RUBblue,left] at (0,2) {$u_e$};
        \node[RUBblue,below] at (2,0) {$u_f$};
        \draw[RUBblue] (0,0) -- (0,2) node[midway, left] {$e$};
        \draw[RUBblue] (0,0) -- (2,0) node[midway, below] {$f$};
    \end{tikzpicture}%
    \caption{%
        A configuration showing that in Case 2,
        the edges
        {\color{RUBblue}$e$} and {\color{RUBblue}$f$}
        cannot lie in  $4$-cycles
        with both
        ${\color{OSNAred}m_1},{\color{OSNAred}m_2}$
        and
        ${\color{RUBgreen}m^\prime_1},{\color{RUBgreen}m^\prime_2}$.%
    }%
    \label{fig:D_eD_f3}%
\end{subfigure}%
\hfill%
\begin{subfigure}{0.45\textwidth}%
    \centering%
    \begin{tikzpicture}[line width=1.5pt, line cap=round, line join=round]
        \draw[OSNAred] (0,0) -- (1.5,1)  node[midway,anchor=-60] {$m_1$};
        \draw[OSNAred] (2,2) -- (1.5,1) node[midway, left] {$m_2$};
        \draw[OSNAgold,dotted] (2,0) -- (2,2) ;
        \draw[OSNAgold,dotted] (0,2) -- (2,2) ;
        \node[RUBblue,anchor=45] at (0,0) {$v = x_1$};
        \node[OSNAgold,right] at (2,2) {$x_2$};
        \node[OSNAgold,left] at (0,2) {$u_e$};
        \node[OSNAgold,below] at (2,0) {$u_f$};
        \draw[RUBblue] (0,0) -- (0,2) node[midway, left] {$e$};
        \draw[RUBblue] (0,0) -- (2,0) node[midway, below] {$f$};
    \end{tikzpicture}%
    \caption{%
        The important configuration from Case 2
        showing that $4$-cycles containing
        {\color{RUBblue}$e$} (respectively {\color{RUBblue}$f$})
        correspond to
        ${\color{OSNAgold}u_e}-{\color{OSNAgold}u_f}$ paths
        of length $2$.%
    }%
    \label{fig:D_eD_f4} %
\end{subfigure}%
%
                \caption{Several configurations that are used in the proof of \Cref{DDbounds_SEP} to bound  $ R_{e,f}$.}%
                \label{fig:D_eD_f}%
            \end{figure}
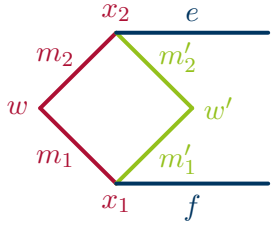
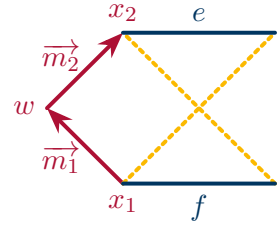
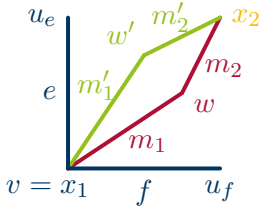
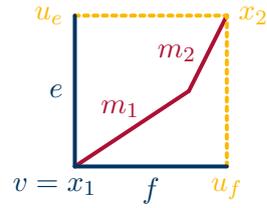%
        
        \end{boundingsubcases}
        
    \item[$T_e^{(f)},T_f^{(e)}$]\label{bound:T-f-e}
     $T_i^{(j)}$ denotes the number of edges of $\SEP_{\Bar{G}+i}$
        that are not edges of $\SEP_{\Bar{G}+i+j}=\SEP_{\Bar{G}+e+f}$ and
        that contain one vertex of the antipodal pair corresponding to $i$. 
        These edges correspond to directed arcs that have exactly one $3$- or $4$-cycle with $i$ in $\Bar{G}+i+j$
        where this unique cycle also contains $j$.
        
        \begin{boundingsubcases}
        
        \item[$e\cap f = \emptyset$]
            As $e$ and $f$ do not share a node by assumption,
            there cannot exist a $3$-cycle containing both.
            Further, there are only $4$ arcs that can lie on a $4$-cycle with both $e$ and $f$,
            with $2$ possible directions each. 
            This almost surely leaves us with $T_e^{(f)},T_f^{(e)} \leq 8$.
            
        \item[$e\cap f=\{v\}$ for some $v\in V$]
            Let $e=u_e v$ and $f=u_f v$.
            As in the previous case, we need to count the directed arcs whose only $3$- or $4$-cycles with $i$ also contain $j$. 
            There can be only one $3$-cycle containing $i$ and $j$, which contributes at most $2$ edges to $T_i^{(j)}$.
            Another arc $m$ can only contribute to $T_i^{(j)}$ if it is part of a $4$-cycle containing both $i$ and $j$ in $\Bar{G}+i+j$.
            This means $m$ has to be part of a $2$-path from $u_e$ to $u_f$.
            As there are two arcs in such a path and two possible orientations we can almost surely bound $T_i^{(j)} \leq 4 W_{n,p}(2)+2$.
        
        \end{boundingsubcases}
        
    \item[$T_{e,f}$]\label{bound:Tef}
        $T_{e,f}$ is the number of edges that are in both, $\SEP_{\Bar{G}+e}$ and $\SEP_{\Bar{G}+f}$, but not in $\SEP_{\Bar{G}+e+f}$.
        These correspond to pairs of directed arcs (not containing $e$ or $f$) that lie on a unique $3$-or $4$-cycle in $\Bar{G}+e+f$ and this one contains both $e$ and $f$.
        Note that, in fact, this unique cycle must be a $4$-cycle, since there are $4$ different arcs involved.
        
        \begin{boundingsubcases}
        
        \item[$e\cap f = \emptyset$]
            Here, a similar argument as for $T_i^{(j)}$ applies.
            There are only two possible $4$-cycles containing $e$ and $f$, with $2$ directions each.
            This gives $T_{e,f}\leq 4$ almost surely. 
            
        \item[$e\cap f=\{v\}$ for some $v\in V$]
            Similar to the discussion for $R_{e,f}$,
            if two or more different pairs of arcs would form a $4$-cycle with $e$ and $f$ in $\Bar{G}+e+f$,
            then these two pairs of arcs would also form a $4$-cycle with each other in $\Bar{G}+e$ and $\Bar{G}+f$.
            The corresponding vertices would then not form an edge in $\SEP_{\Bar{G}+e}$ nor $\SEP_{\Bar{G}+f}$.
            Therefore, there can only be at most one relevant pair of arcs in $\Bar{G}$,
            leading to the almost sure bound $T_{e,f}\leq 2$,
            taking into account the two orientations.
            
        \end{boundingsubcases}
        
    \standarditem[Conclusion.]
        Putting these bounds together yields that almost surely
        \begin{align*}
            \vert \D_e \D_f K \vert
            &\leq  pq (4+8+8+8+4)
            = 32pq
            \simeq pq
        \end{align*}
        for $e\cap f = \emptyset$,
        and almost surely
        \begin{align*}
            \vert \D_e \D_f K \vert
            &\leq  pq \Bigl(
                4
                + \bigl( 2 W_{n,p}(2) + 4 \bigr)
                + \bigl( 4 W_{n,p}(2) + 2 \bigr)
                + \bigl( 4 W_{n,p}(2) + 2 \bigr)
                + 2
                \Bigr)
            \\
            &= pq \bigl( 10 W_{n,p}(2) + 8 \bigr)
        \end{align*}
        for $\vert e\cap f\vert= 1$.
        Thus, in the latter case, by \Cref{l:W123_bound_sep} we obtain
        $$
            \E\left[(\D_e\D_f K)^4\right]
            \lesssim 
            p^4q^4
            \begin{cases}
                n^4 p^8, & \text{if } p \gg n^{-1/2}, \\
                1,       & \text{if } p \lesssim n^{-1/2}.
            \end{cases}
        $$
        This completes the proof.
        \qedhere
    
    \end{boundingcases}
    
\end{proof}

\subsection{Central limit theorem}

Bringing the results developed so far in this section together yields the following central limit theorem for the number of edges in the random symmetric edge polytope.

\begin{thm}\label{CLT_SymmetricEdge}
    Let \( K_{n,p} \) be the number of edges of \( \SEP_{n,p} \), and \( N \) be a standard Gaussian random variable.
    If \( p \gg n^{-2} \), then
    \[
        d_K\left(\frac{K_{n,p}-\EE[K_{n,p}]}{\sqrt{\var(K_{n,p})}}, N\right)
        \lesssim
        \min \Bigl\{
            \frac{1}{n\sqrt{pq}(1-\sqrt{2}p)^2}
            \;,\;
            \frac{1}{\sqrt{pq}}
        \Bigr\}.
    \]
    In particular, if additionally \( q \gg n^{-2} \) and \( |p-\frac{1}{\sqrt{2}}| \gg n^{-1/2} \), then \( K_{n,p} \) satisfies a central limit theorem as \( n \to \infty \).
\end{thm}
\begin{proof}[Proof of \Cref{CLT_SymmetricEdge}]
    We apply \Cref{prop:2ndOrderPoincare} and bound the terms \( B_1, \dots, B_5 \) using the moment estimates for discrete gradients derived in \Cref{l:W123_bound_sep} and \Cref{DDbounds_SEP}.
    Recall that for \( p \gg n^{-2} \), we have \( \EE[(\D_e K_{n,p})^4] \lesssim n^8 p^6 q^2 \).
    By the Cauchy--Schwarz inequality, for any \( e, f \in \binom{[n]}{2} \),
    \begin{align*}
        \EE\bigl[ (\D_e K_{n,p})^2 (\D_f K_{n,p})^2 \bigr]
        &\leq \sqrt{ \EE[(\D_e K_{n,p})^4] \cdot \EE[(\D_f K_{n,p})^4] }
        = \EE[(\D_e K_{n,p})^4]
        \lesssim n^8 p^6 q^2.
    \end{align*}
    Similarly, for second-order gradients, \Cref{DDbounds_SEP} implies for any \(e,g\in \binom{[n]}{2}\)
    \[
        \EE\bigl[ (\D_g \D_e K_{n,p})^4 \bigr] \lesssim
        \begin{cases}
            p^4 q^4, & \text{if } e \cap g = \emptyset, \\
            n^4 p^{12} q^4 + p^4 q^4, & \text{if } |e \cap g| = 1.
        \end{cases}
    \]
    Applying Cauchy--Schwarz to the product of second-order gradients yields
    \begin{align*}
        \EE\bigl[ (\D_g \D_e K_{n,p})^2 (\D_g \D_f K_{n,p})^2 \bigr]
        \lesssim
        \begin{cases}
            p^4 q^4, & \text{if } |g \cap e| = |g \cap f| = 0, \\
            n^2 p^8 q^4 + p^4 q^4, & \text{if } |g \cap e| + |g \cap f| = 1, \\
            n^4 p^{12} q^4 + p^4 q^4, & \text{if } |g \cap e| = |g \cap f| = 1.
        \end{cases}
    \end{align*}
    We now estimate the terms \( B_1, \dots, B_5 \) appearing in \Cref{prop:2ndOrderPoincare}. To determine the order of magnitude of these sums, we distinguish cases based on the intersection of the arc indices. We focus on the terms with the highest powers of \( n \), as these dominate the sums.

    \begin{boundingcases}
    
    \item[\( B_1 \)]
        The term \( B_1 \) captures the interaction between first-order gradients mediated by second-order gradients:
        \[
            B_1 = \sum_{e,f,g \in \binom{[n]}{2}} \sqrt{ \EE[ (\D_e K)^2 (\D_f K)^2 ] } \sqrt{ \EE[ (\D_g \D_e K)^2 (\D_g \D_f K)^2 ] }.
        \]
        Using the Cauchy--Schwarz inequality, we bound the expectations by the square root of the fourth moments. We split the sum based on the overlap of \( g \) with \( e \) and \( f \).
    
        \begin{boundingsubcases}
        
        \standarditem[Disjoint case (\( |g \cap e| = 0 \) and \( |g \cap f| = 0 \)).]
            There are of order \( n^6 \) such triplets. The fourth moment of the second-order gradient is bounded by \( \lesssim p^4 q^4 \). Thus, the summand is bounded by
            \[
                \sqrt{n^8 p^6 q^2 \cdot n^8 p^6 q^2} \cdot \sqrt{p^4 q^4 \cdot p^4 q^4} = n^8 p^6 q^2 \cdot p^4 q^4 = n^8 p^{10} q^6.
            \]
            The total contribution is \( n^6 \cdot n^8 p^{10} q^6 = n^{14} p^{10} q^6 \).
            
        \standarditem[Overlap case (\( |g \cap e| = 1 \) or \( |g \cap f| = 1 \)).]
            The dominant contribution comes from the case where \( g \) shares nodes with both \( e \) and \( f \) (e.g., \( g \) connects \( e \) and \( f \)). There are of order \( n^4 \) such terms. Using the bound \( \EE[(\D_g \D_e K)^4] \lesssim n^4 p^{12} q^4 \), the summand is bounded by
            \[
                \sqrt{n^8 p^6 q^2 \cdot n^8 p^6 q^2} \cdot \sqrt{n^4 p^{12} q^4 \cdot n^4 p^{12} q^4} = n^8 p^6 q^2 \cdot n^4 p^{12} q^4 = n^{12} p^{18} q^6.
            \]
            The total contribution is \( n^4 \cdot n^{12} p^{18} q^6 = n^{16} p^{18} q^6 \).
            
        \end{boundingsubcases}
    
        Thus, \( B_1 \lesssim n^{16} p^{18} q^6 + n^{14} p^{10} q^6 \).
    
    \item[\( B_2 \)]
        Recall that
        \[
            B_2 = \sum_{e,f,g\in\binom{[n]}{2}} \frac{1}{pq} \EE[ (\D_g \D_e K)^2 (\D_g \D_f K)^2 ].
        \]
        
        \begin{boundingsubcases}
        
        \standarditem[Disjoint case (\( |g \cap e| = 0, |g \cap f| = 0 \)).]
            There are of order \( n^6 \) such triplets. The discrete gradients are bounded by \( p^4 q^4 \). So, the contributions of these terms is
            \[
                \lesssim \frac{n^6 \cdot \sqrt{p^4 q^4 \cdot p^4 q^4}}{pq} = n^6 p^3 q^3.
            \]
            
        \standarditem[Full Overlap (\( |g \cap e| = 1, |g \cap f| = 1 \)).]
            There are of order \( n^4 \) such triplets. The discrete gradients are bounded by \( n^4 p^{12} q^4 \). It follows that the contributions of these terms satisfies
            \[
                 \lesssim \frac{n^4 \cdot \sqrt{n^4 p^{12} q^4 \cdot n^4 p^{12} q^4}}{pq} = \frac{n^4 \cdot n^4 p^{12} q^4}{pq} = n^8 p^{11} q^3.
            \]
            
        \end{boundingsubcases}
        
        Thus, \( B_2 \lesssim n^8 p^{11} q^3 + n^6 p^3 q^3 \).
    
    \item[\( B_3 \)]
        The term \( B_3 \) involves only the fourth moments of the first-order discrete gradients:
        \[
            B_3 = \sum_{e \in \binom{[n]}{2}} \frac{1}{pq} \EE[(\D_e K)^4].
        \]
        There are \(\binom{n}{2} \simeq n^2\) terms in this sum. Using the bound \( \EE[(\D_e K)^4] \lesssim n^8 p^6 q^2 \) derived in \Cref{l:W123_bound_sep}, we calculate the total contribution as
        \[
            B_3 \lesssim \frac{n^2 \cdot n^8 p^6 q^2}{pq} = n^{10} p^5 q.
        \]
    
    \item[\( B_4 \)]
        This term couples first-order and second-order gradients:
        \[
            B_4 = \sum_{e,g\in\binom{[n]}{2}} \frac{1}{pq} \sqrt{\EE[(\D_e K)^4]} \sqrt{\EE[(\D_g \D_e K)^4]}.
        \]
        We distinguish between the cases where the arcs \( e \) and \( g \) are disjoint or share a node.
        
        \begin{boundingsubcases}
        
        \standarditem[Disjoint case (\( |e \cap g| =0 \)).]
            There are of order \( n^4 \) such pairs. The first-order gradient contributes \( \sqrt{n^8 p^6 q^2} = n^4 p^3 q \), while the second-order gradient is bounded by \( \sqrt{p^4 q^4} = p^2 q^2 \). Thus, these terms contribute
            \[
                \lesssim \frac{n^4 \cdot n^4 p^3 q \cdot p^2 q^2}{pq} = n^8 p^4 q^2.
            \]
            
        \standarditem[Overlap case (\( |e \cap g| = 1 \)).]
            There are of order \( n^3 \) such pairs. The first-order gradient remains \( n^4 p^3 q \), but the second-order gradient now contributes \( \sqrt{n^4 p^{12} q^4} = n^2 p^6 q^2 \). The contribution is
            \[
                \lesssim \frac{n^3 \cdot n^4 p^3 q \cdot n^2 p^6 q^2}{pq} = n^9 p^8 q^2.
            \]
            
        \end{boundingsubcases}
        
        Thus, \( B_4 \lesssim n^9 p^8 q^2 + n^8 p^4 q^2 \).
    
    \item[\( B_5 \)]
        Finally, \( B_5 \) involves only the fourth moments of the second-order gradients:
        \[
            B_5 = \sum_{e,g\in\binom{[n]}{2}} \frac{1}{(pq)^2} \EE[(\D_g \D_e K)^4].
        \]
        Again, we split the sum  based on the intersection of \( e \) and \( g \).
    
        \begin{boundingsubcases}
        
        \standarditem[Disjoint case (\( |e \cap g| = 0 \)).]
            There are of order \( n^4 \) terms. Using the bound \( \EE[(\D_g \D_e K)^4] \lesssim p^4 q^4 \), we obtain a contribution of
            \[
                \lesssim \frac{n^4 \cdot p^4 q^4}{p^2 q^2} = n^4 p^2 q^2.
            \]
            
        \standarditem[Overlap case (\( |e \cap g| = 1 \)).]
            There are of order \( n^3 \) terms. Using the larger bound \( \EE[(\D_g \D_e K)^4] \lesssim n^4 p^{12} q^4 \), the contribution is
            \[
                \lesssim \frac{n^3 \cdot n^4 p^{12} q^4}{p^2 q^2} = n^7 p^{10} q^2.
            \]
            
        \end{boundingsubcases}
        
        Thus, \( B_5 \lesssim n^7 p^{10} q^2 + n^4 p^2 q^2 \).
    
    \standarditem[Conclusion.]
        We denote the Kolmogorov distance between the standardized random variable and the standard normal distribution by
        \[
            d_K \coloneqq d_K\left(\frac{K_{n,p}-\EE[K_{n,p}]}{\sqrt{\var(K_{n,p})}}, N\right).
        \]
        According to \Cref{prop:2ndOrderPoincare}, this distance satisfies \( d_K \lesssim \sum_{i=1}^5 \sqrt{B_i} \).
        Substituting the bounds derived above into this inequality, we obtain
        \begin{align*}
            d_K
            &\lesssim \sqrt{B_1} + \sqrt{B_2} + \sqrt{B_3} + \sqrt{B_4} + \sqrt{B_5}
            \\
            &\lesssim \frac{\sqrt{
                n^{10} p^5 q^3 +
                n^6 p^3 q^3 +
                n^8 p^{11} q^3 +
                n^{10} p^5 q +
                n^8 p^4 q^2 +
                n^9 p^8 q^2 +
                n^4 p^2 q^2 +
                n^7 p^{10} q^2
            }}{\var(K_{n,p})}
            \\
            &\simeq \frac{ n^{5} p^{\frac52} q^{\frac12} }{ \var(K_{n,p}) }.
        \end{align*}
        Here, we utilized that the term \( n^{10} p^5 q \) originating from \( B_3 \) dominates all other terms in the numerator under the assumption \( p \gg n^{-2} \).
        Inserting the lower bound for the variance from \Cref{variance_sep}, given by \( \var(K_{n,p}) \simeq n^6 p^3 q (1-\sqrt{2}p)^2 + n^5 p^3 q \), yields
        \begin{align*}
            d_K
            &\lesssim \frac{ n^{5} p^{\frac52} q^{\frac12} }{ n^6 p^3 q \bigl( 1 - \sqrt{2}p \bigr)^2 + n^5 p^3 q}
            \simeq \frac{1}{ n \sqrt{pq} (1-\sqrt{2}p)^2 + \sqrt{pq} }.
        \end{align*}
        Finally, observing that the reciprocal of a sum is comparable to the minimum of the reciprocals, we arrive at
        \[
            d_K \lesssim \min \Bigl\{ \frac{1}{ n \sqrt{pq} (1-\sqrt{2}p)^2 } , \frac{1}{\sqrt{pq}} \Bigr\}.
        \]
        This bound implies that \( d_K \to 0 \) as \( n \to \infty \) provided that \( p,q \gg n^{-2} \) and \( |p-\frac{1}{\sqrt{2}}| \gg n^{-1/2} \), which establishes the central limit theorem.
        \qedhere
    
    \end{boundingcases}

\end{proof}

\begin{rem}
    \Cref{CLT_SymmetricEdge} establishes asymptotic normality for the standardized edge count of $\SEP_{n,p}$ provided we are away from the critical probability \( p = 1/\sqrt{2} \).
    From \Cref{variance_sep}, we know that if \( p \lesssim n^{-2} \) or \( q \lesssim n^{-6} \), the variance does not diverge, precluding a central limit theorem. In this regime, the leading \(n^6\)-term in the variance is suppressed, and \(\var(K_{n,p})\) is only of order \(n^5\). The difficulty is that our present Malliavin--Stein estimates rely on coarse bounds for the first-order discrete gradients, which do not capture the cancellation mechanism responsible for this variance reduction. As a consequence, the resulting error terms are too large to conclude asymptotic normality in the critical window. We nevertheless expect that the central limit theorem remains valid also in this regime. Establishing this would require a more refined analysis of the difference operators, reflecting the additional cancellation at \( p=\frac{1}{\sqrt{2}} \), and would presumably lead to a slower rate of convergence.
\end{rem}

\section[Proof of Theorem 1.1 for triangulations of symmetric edge polytopes]{Proof of \Cref{thm:Main} for Triangulations of Symmetric Edge Polytopes}\label{sec:TriangulationsSEP}

The goal of this section is to provide the proof of \Cref{thm:Main} for the number of edges of a unimodular triangulation $\Delta_{n,p}$ of the random symmetric edge polytope \(\SEP_{n,p}\). We note that this number does not depend on the particular choice of the triangulation, and hence is a well-defined quantity. Accordingly, in the following we use \(K_{n,p}\) to denote this common number of edges of any unimodular triangulation of \(\SEP_{n,p}\).

The proof will closely follow the proof of \Cref{thm:Main} for the number of edges of $\SEP_{n,p}$ and we will be able to use most of our estimates from \Cref{sec:SEP-proof} with small adjustments to account for the additional edges added by the triangulation.

\subsection{Expectation}
We start by analyzing the expectation $\E(K_{n,p})$. 

\begin{thm}\label{lem:ExpectationTriangulation}
    Let \(K_{n,p}\) denote the number of edges in any unimodular triangulation $\Delta_{n,p}$ of \(\SEP_{n,p}\).
    Then,
    \begin{align*}
        \E(K_{n,p})
        &= 6 \binom{n}{4} p^2 (2-p^2)
            + 6 \binom{n}{3} p^2
            + 2 \binom{n}{2} (1-p) \left(1-(1-p^2)^{n-2}\right)
            + 2 \binom{n}{2} p^2
        .
    \end{align*}
    As \(n \to \infty\), this behaves like $\EE[K_{n,p}]\;\simeq\;n^4 p^2$.
\end{thm}
\begin{proof}
    The proof follows arguments similar to those used in the proof of \Cref{expectation_sep}. However, in the present setting, for a pair of arcs \(e,f\in \binom{[n]}{2}\), we do not need to consider \(4\)-cycles in which \(e\) or \(f\) is the smallest arc with respect to a fixed ordering \(w\).
Let \(\arr{e}\) and \(\arr{f}\) be two oriented arcs, and let \(e\) and \(f\) denote their unoriented counterparts. We distinguish cases according to how \(e\) and \(f\) intersect, and we conclude with the case in which one endpoint of the triangulation edge is the origin.

\begin{maincases}

\item[$e\cap f=\emptyset$]
In this case, there is a unique directed \(4\)-cycle \(C\) whose set of arcs contains \(\arr e\) and \(\arr f\).
The same cycle \(C\) is also the unique \(4\)-cycle containing the remaining two (directed) arcs of \(C\).
Consequently, among all pairs \((\arr e,\arr f)\) with \(e\cap f=\emptyset\), exactly half of them contain the smallest arc of \(C\) with respect to the fixed ordering \(w\) and the number of such pairs equals \(6\binom{n}{4}\).
Moreover, the corresponding vertices form an edge of \(\Delta_{n,p}\) if and only if \(e,f\in G_{n,p}\), which yields the contribution
\[
6\binom{n}{4}\,p^2
\]
to the expectation.

If neither \(\arr e\) nor \(\arr f\) is the smallest arc of \(C\), then the corresponding vertices form an edge of \(\Delta_{n,p}\) with probability \(p^2(1-p^2)\), as in \Cref{expectation_sep}. These terms contribute
\[
6\binom{n}{4}\,p^2(1-p^2)
\]
to $\E(K_{n,p})$. 
\item[$e\cap f=\{v\}$ for some \(v\in V\)]
As for the symmetric edge polytope, we need to analyze which directed \(3\)- or \(4\)-cycles \(\arr e\) and \(\arr f\) may lie on. We distinguish cases according to the orientations at the common node \(v\).

\begin{subcases}

\item
If both arcs point towards \(v\), or both point away from \(v\), then \(\arr e\) and \(\arr f\) cannot lie on a directed cycle. As in \Cref{expectation_sep}, this yields the contribution
\[
6\binom{n}{3}\,p^2.
\]

\item
Assume that \(\arr e\) and \(\arr f\) are oriented in opposite directions at \(v\). Then they lie on a unique directed \(3\)-cycle, which is completed by a unique additional arc \(g=u_1u_2\). We fix such an arc \(g\) and count the pairs that form a directed \(3\)-cycle together with \(g\).

There are \(n-2\) pairs of arcs \(e_i=u_1w_i\) and \(f_i=u_2w_i\), \(i\in[n-2]\), one for each node \(w_i\in V\setminus\{u_1,u_2\}\), that form a directed \(3\)-cycle with \(g\).
Moreover, for \(i\neq j\), the two pairs \((e_i,f_i)\) and \((e_j,f_j)\) together form a \(4\)-cycle. Since at most one of these pairs can contain the smallest arc of that \(4\)-cycle, it follows that \(\Delta_{n,p}\) contains at most one such edge pair per choice of \(g\), and there are two possible orientations. Such an edge pair occurs precisely if \(g\notin G_{n,p}\), which has probability \(1-p\), and at least one of the pairs \((e_i,f_i)\) is present in \(G_{n,p}\), which has probability \(1-(1-p^2)^{n-2}\).
Since there are \(\binom{n}{2}\) choices for \(g\), this case contributes
\[
2\binom{n}{2}\,(1-p)\Bigl(1-(1-p^2)^{n-2}\Bigr).
\]

\end{subcases}

\item[One endpoint corresponds to the origin]
Since \(\Delta_{n,p}\) contains an edge between the origin and every vertex of \(\SEP_{n,p}\), and since \(\SEP_{n,p}\) has two vertices per arc \(e\in G_{n,p}\), this case contributes
\[
2\binom{n}{2}\,p.
\]
\end{maincases}
Summing the contributions from all cases yields the result.    
\end{proof}

\subsection{Variance}

Turning to the variance of $K_{n,p}$, our approach closely mirrors the proof of \Cref{variance_sep}. Given the intricate nature of the triangulations, however, we establish only a lower bound for the variance. This bound is precisely what is required for the subsequent proof of the central limit theorem.

\begin{thm}\label{variance_sepTriag}
    Let $K_{n,p}$ denote the number of edges of a unimodular triangulation $\Delta_{n,p}$ of $\SEP_{n,p}$ and assume $\limsup p < 1$. Then   
    \begin{align*}
        \VV(K_{n,p}) \gtrsim \begin{cases}
            n^6p^3,  &\text{ if } p\gg \frac{1}{n^2}\\
            n^2p, &\text{ if } p\lesssim \frac{1}{n^2}.
        \end{cases}
    \end{align*}
\end{thm}

\begin{proof} 
    Following the strategy from the proof of \Cref{variance_sep}, let $\indicator_{\{\arr{e},\arr{f}\}}$ denote the event that the vertices corresponding to $\arr{e}$ and $\arr{f}$ form an edge. We analyze the interactions between two pairs of arcs, $e, f$ and $g, h$, using the same case distinctions as before (illustrated in \Cref{fig:variance}). Finally, we incorporate the cases where one of the vertices is the origin.

    In contrast to the previous proof, we do not need to consider $4$-cycles in which one arc of our pair is the minimal element in a fixed term order. To simplify the counting argument, we introduce a random term order. Let $w \in S_{\binom{n}{2}}$ be a uniformly sampled linear order of the potential arcs of $G$, chosen independently of $G_{n,p}$. Note that $w$ naturally induces a linear order on the arcs of $G_{n,p}$ by restriction. For any linear order $w$, let $K_{n,p}^{(w)}$ denote the number of edges of a triangulation $\Delta$ induced by $w$. Since the total value of $K_{n,p}$ does not depend on the specific term order, and since $w$ is independent of the underlying graph, the law of total variance yields
\begin{align*}
    \V(K_{n,p}) 
    &= \E\bigl[\V(K_{n,p}^{(w)} \mid w)\bigr] = \sum_{\{\arr{e},\arr{f}\}}
        \sum_{\{\arr{g},\arr{h}\}} 
        \E\bigl[
            \mathrm{Cov}\bigl(\indicator_{\{\arr{e},\arr{f}\}}, \, \indicator_{\{\arr{g},\arr{h}\}} \mid w\bigr)
        \bigr].
\end{align*}
This formulation allows us to reframe the problem: instead of explicitly counting configurations where certain arcs are minimal, we can compute the probabilities of these minimal elements occurring within the small subgraphs associated with each case. Moreover, cases yielding independent random variables in the proof of \Cref{variance_sep} remain independent here, since introducing a term order merely eliminates certain configurations. We refer to \Cref{variance_sep} for a detailed breakdown of the case distinction, general strategy, and notation.

    \begin{maincases}
    
    \item[$e\cap f=\emptyset$ and $g\cap h=\emptyset$]
        Notice that in this case there is only one $4$-cycle containing each pair. 
        If one of the arcs is the minimal arc of this $4$-cycle, they are connected by an edge regardless of whether the $4$-cycle is contained in $G_{n,p}$, so $\indicator_{\{\arr{e},\arr{f}\}}$ is just the indicator that both arcs are contained in $G$. 

        \begin{subcases}
            
        \item[\( \vert (e \cup f) \cap (g \cup h) \vert \leq 1 \)]
            In this case, as in the proof of \Cref{variance_sep},
            \(\indicator_{\{\arr{e},\arr{f}\}}\)
            and
            \(\indicator_{\{\arr{g},\arr{h}\}}\)
            are independent random variables. 
        
        \item[\( \vert (e \cup f) \cap (g \cup h) \vert = 2 \)]\label{subcase:SEPTriagvariance:12}
            The assumption that the pairs share two nodes leads to the three configurations  that are depicted in
            \nameCref{fig:variance}~\hyperref[fig:variance]{2b}. 
        
            In the first configuration there are two paths of length two, each consisting of one arc from $\{e,f\}$ and one arc from $\{g,h\}$.
            The variables \(\indicator_{\{\arr{e},\arr{f}\}}\) and \(\indicator_{\{\arr{g},\arr{h}\}}\) are dependent, with covariance $p^7(1-p)$, if and only if the $4$-cycle of both pairs uses the arc $k$ between the vertices $e\cap g$ and $f\cap h$ and neither pair contains the smallest arc of their respective $4$-cycle.
            To count in how many cases neither pair contains the smallest arc of their $4$-cycle we look at the smallest arc among the $7$ arcs contained in the $4$-cycles.
            If one of $e,f,g,h$ is the smallest arc, we know it is the smallest arc of its $4$-cycle, so \(\indicator_{\{\arr{e},\arr{f}\}}\) and \(\indicator_{\{\arr{g},\arr{h}\}}\) are independent.
            If $k$ is the smallest arc, it is in particular the smallest among both $4$-cycles, so the random variables are dependent, this happens in $\frac{1}{7}$ of the cases.
            In the case where one of the two unnamed arcs is the smallest among the $7$ arcs, we do not know yet which arc is the smallest in the other $4$-cycle.
            As there are $4$ arcs in the other one, two of which correspond to the cycle depending on $k$, this leads to additional $\frac{2}{7}\cdot \frac{1}{2}=\frac{1}{7}$ of the cases where the arcs are dependent.
            Concluding this case, we see that the contribution to the covariance is $\frac{2}{7}$ times the contribution we have seen above, yielding
            \[
                \frac{2}{7} \cdot 2 (n)_6 p^7(1-p) = \frac{4}{7}(n)_6 p^7(1-p).
            \]
        
            The second configuration consists of one path of length $3$ and an additional non-incident arc, where the extra arc is in the same pair as the middle arc of the path. 
            Using the labeling of the second configuration in  \nameCref{fig:variance}~\hyperref[fig:variance]{2b}, the dependency here exists if the middle arc $e$ is relevant for \(\indicator_{\{\arr{g},\arr{h}\}}\). 
            This happens if and only if neither $g$ nor $h$ is the smallest arc of its $4$-cycle (and $g$ and $h$ are oriented properly), which is half of the cases as in the proof of \Cref{variance_sep}. 
            This means that we have to consider $2 (n)_6$ cases total, which we will split further based on what the covariance is.
            
            If $e$ or $f$ is the smallest arcs of their cycle we obtain a covariance of
            $$
                p^4(1-p) - p^4(1-p^2) = - p^5(1-p).
            $$
            This can happen is three ways. 
            If $e$ is the smallest occurring arc the condition is met for both pairs, yielding a contribution of $\frac{1}{7}$ of the total cases.  
            If $f$ is the smallest occurring arc, we need to make sure that neither $g$ nor $h$ is the smallest arc of their cycle, yielding $\frac{1}{2}\cdot \frac{1}{7}=\frac{1}{14}$ of the total cases. 
            If the remaining arc of the $4$-cycle of $g,h$ is the smallest, we need to make sure that either $e$ or $f$ is the smallest arc of their respective $4$-cycle, which happens in $\frac{1}{7}\cdot\frac{1}{2}=\frac{1}{14}$ of the total cases.
            Altogether this yields $\frac{4}{14}$ of the total cases, which is $\frac{4}{7}$ of the cases where $g,h$ are not the smallest arcs of their respective $4$-cycles.
            
            If neither $e$ nor $f$ is the smallest arc of their cycle we obtain the covariance from \Cref{variance_sep}, which is
            $$
                -p^5(1-p)(1-p^2),
            $$
            which happens in the remaining $\frac{3}{7}$ of cases.
            
            This yields an overall contribution of 
            \begin{align*}
               &{} -2(n)_6 p^5(1-p) \left(\frac{4}{7}+\frac{3}{7}(1-p^2)\right) = -(n)_6p^5(1-p) (2-\frac{6}{7}p^2).
            \end{align*}
            
            Finally if the two pairs share one arc, as depicted in the third configuration of \Cref{fig:variance}, the pairs are always dependent. We split the $2(n)_6$ cases further based on how many pairs contain the smallest arc of their $4$-cycle. 
            Let $J$ be a set of four appearing arcs besides, $e=g,f,h$.
        
            If $J$ contains the smallest arc of both cycles we obtain a covariance of 
            $$
                p^3(1-p)(1-p^2)^2
            $$
            as in the proof of \Cref{variance_sep}.
            This happens if one arc of $J$ is the smallest among all $7$ arcs and additionally one of the two arcs in $J$ is the smallest of the other cycle. 
            This yields $\frac{4}{7}\cdot \frac{1}{2}= \frac{2}{7}$ of the cases.
        
            If one of the pairs contains the smallest arc of their $4$-cycle we need to exclude one $4$-cycle, so we obtain a covariance of 
            $$
                p^3(1-p)(1-p^2).
            $$
            One way for this to happen is the other half of the case counted above, i.e., one arc of $J$ is the smallest and one arc of $e,f,g,h$ is the smallest in the other cycle, which yields $\frac{2}{7}$ of the cases.
            Another case occurs if one arc of $f,h$ is the smallest and one arc of $J$ is the smallest among the other cycle, which happens in $\frac{2}{7}\cdot \frac{1}{2}=\frac{1}{7}$ of the cases, which leads to $\frac{3}{7}$ of the cases. 
        
            If $J$ does not contain the smallest arc of any cycle, we do not need to exclude any other arcs, which yields a covariance of 
            $$
            p^3-p^4 = p^3(1-p) .
            $$
            This case happens if $e=g$ is the smallest arc, or if the smallest arc is among $f,h$ and the smallest arc in the remaining cycle is not in $J$. Altogether this yield $\frac{1}{7}+\frac{2}{7}\cdot \frac{1}{2}=\frac{2}{7}$ of the cases.
            
            Collecting the terms from this configuration yields 
            $$
               2(n)_6 \biggl(\frac{2}{7}p^3(1-p)(1-p^2)^2+ \frac{3}{7}p^3(1-p)(1-p^2)+ \frac{2}{7}p^3(1-p)\biggr)
               = 2(n)_6(1-p)p^3\biggl(1- p^2+ \frac{2}{7}p^4\biggr),
            $$
            and combining all contribution from \Cref{subcase:SEPTriagvariance:12}, yields
            \begin{align*}
                & \frac{4}{7} (n)_6 p^7(1-p)
                    - (n)_6 p^5 (1-p) \biggl( 2-\frac{6}{7}p^2 \biggr)
                    + 2 (n)_6 p^3 (1-p) \biggl( 1-p^2 + \frac{2}{7}p^4 \biggr)
                \\
                ={}& (n)_6 p^3 \Biggl(
                        \frac{4}{7} p^4 (1-p)
                        - p^2 (1-p) \biggl( 2-\frac{6}{7} p^2 \biggr)
                        + 2 (1-p) \biggl( 1-p^2 + \frac{2}{7} p^4 \biggr)
                    \Biggr)
                \\
                ={}& 2 (n)_6 p^3 (1-p) \bigl( 1 - p^2 \bigr)^2
                ,
            \end{align*}
            which is going to be our target rate in the regime $p\gg \frac{1}{n^2}$.
            
        \item[\(\vert (e \cup f) \cap (g \cup h) \vert = 3\)]
            All configurations here use exactly $5$ nodes and therefore at least $3$ arcs. Therefore these cases can contribute at most on the order of $n^5p^3$, which is dominated by our target rate $n^6p^3$ in all regimes.
        
        \item[\(\vert (e \cup f) \cap (g \cup h) \vert = 4\)]
            All configurations in this case use four nodes and at most $2$ arcs. Therefore they can contribute at most on the order of $n^4p^2$, which is dominated by $n^6p^3$ if $p\gg \frac{1}{n^2}$ and does not exceed $n^2p$ in the regime $p\lesssim \frac{1}{n^2}$.
        
        \end{subcases}
        
    \item[$\vert e \cap f \vert = 1$ or $\vert g \cap h \vert = 1$]
    
        \begin{subcases}
    
        \item[$\vert e \cup f \cup g \cup h \vert = 7 $]
            As argued in the proof of \Cref{variance_sep}, the indicator variables in this case are independent. 
            
        \item[$\vert e \cup f \cup g \cup h \vert = 6 $]
            As the first configuration of \nameCref{fig:variance}~\hyperref[fig:variance]{2f} yields an independent case, there are two configurations we need to consider.
            
            The first consist of one path of length $3$, where two consecutive arcs are part of the same pair, and an isolated arc, as shown in \nameCref{fig:variance}~\hyperref[fig:variance]{2f}. 
            We will show that this case actually contributes only on the order of $n^5p^3$ and thus is dominated by $n^6p^3$ in all regimes.
            The indicator variables here are dependent because, using the notation from the figure, one $4$-cycle $C$ that contains $e$ and $f$ also uses $h$. 
            Note that the unique $4$-cycle containing $g$ and $h$ does not share any arcs with $e,f$ or their $3$- or $4$-cycles.
            If $e$ or $f$ is the smallest arc of $C$, their indicator does not depend on arc $h$, so we get independence. 
            On the other hand, if neither $e$ nor $f$ is the smallest arc of $C$, one of the $4$-cycles relevant for $ \indicator_{\{\arr{e},\arr{f}\}} $ is missing only one arc.
            If there are $\ell$  $4$-cycles where the smallest element is neither $e$ nor $f$ for some $\ell\in [n-4]$, as we already know $h$ is contained in a smaller pair, this induces a covariance of
            $$
                p^4\left( (1-p)^2(1-p^2)^{\ell-2} - (1-p)(1-p^2)^{\ell-1}\right) = -p^5(1-p)^2(1-p^2)^{\ell-2}
            $$
            if $g$ or $h$ is the smallest arc of their $4$-cycle and a covariance of 
            \begin{align*}
                &
                    p^4 \left( (1-p^2)(1-p) (1-p) (1-p^2)^{\ell-2}
                    - (1-p^2)(1-p)(1-p^2)^{\ell-1} \right)
                =
                    -p^5(1-p)^2(1-p^2)^{\ell-1}
            \end{align*}
            if neither $g$ or $h$ is the smallest arc of their $4$-cycle. Note that as the latter term is smaller in absolute value, we can bound the covariance in both cases from below by
            $$
                -p^5(1-p)^2(1-p^2)^{\ell-2}.
            $$
            To count the sum over these configurations, note that $e,f$ have a unique arc $k$ they form a $3$-cycle with.
            For each such arc $k\in \binom{[n]}{2}$ there exists exactly one pair $e_\ell,f_\ell$ that is the $\ell$\textsuperscript{th} smallest pair of its $4$-cycles for each $\ell\in\{1,\dots,n-2\}$.
            Fixing a pair $e,f$ as the $\ell$\textsuperscript{th} smallest pair, one has $2(\ell-1)$ ways to choose arc $h$, as it has to belong to a smaller pair and each pair consists of $2$ arcs. 
            One has $\binom{n-4}{2}$ possibilities to choose arc $g$. 
            Finally there are $2$ ways to direct $e,f$ and $4$ ways to direct $g,h$. 
            This yields a total of
            \begin{align*}
                &
                    \binom{n}{2} \binom{n-4}{2} 4
                    \cdot 2
                    \cdot (-p^5)(1-p)^2
                    \sum_{\ell=2}^{n-2}\bigl( 1-p^2 \bigr)^{\ell-2}2(\ell-1)
                .
            \end{align*}
            Note that
            \begin{align*}
                &
                    \sum_{\ell=2}^{n-2}(1-p^2)^{\ell-2}(\ell-1)
                =
                    \sum_{\ell=0}^{n-4} \sum_{k=0}^{\ell} (1-p^2)^{\ell}
                =
                    \sum_{k=0}^{n-4} \sum_{\ell=k}^{n-4} (1-p^2)^{\ell}
                \\
                ={}&
                   \sum_{k=0}^{n-4} (1-p^2)^k \frac{1-\bigl( 1-p^2 \bigr)^{n-4+1-k}}{1-\bigl( 1-p^2 \bigr)}
                \leq
                    \sum_{k=0}^{n-4} 1 \cdot \frac{1}{1-\bigl( 1-p^2 \bigr)}
                =
                    \frac{n-3}{p^2}
            \end{align*}
            so that the total contribution from this case
            is at most of order
            \begin{align*}
                    \binom{n}{2} \binom{n-4}{2} 4
                    \cdot 2
                    \cdot (-p)^5 (1-p)^2
                    \cdot \frac{n-3}{p^2}
                \simeq{}&
                    n^5 p^3 (1-p)^2
                ,
            \end{align*}
            which is smaller than $n^6p^3$ in all regimes.
        
            The other case consists of two disjoint pairs of $2$-paths, where each path corresponds to one pair. 
            Our goal here will be to show that this case has a positive contribution and as such does not cancel out our positive variance of $n^6p^3$.
            In this case both pairs have $4$-cycles that use the arcs connecting the leafs of one path with the leafs of the other one. The dependence here arises if both pairs use the same of these connecting arcs.
            This also means that if one pair is directed such that it always forms an edge or one of $e,f,g,h$ is smaller than the connecting arcs, the random variables are independent.
            So that both $e,f$ and $g,h$ do not contain the smallest arc of their $4$-cycles and they are directed such they can form a $4$-cycle.
            In this case both pairs form an arc if of the $4$ connecting arcs, either none or one are part of $G_{n,p}$ or if the two connecting arcs part of $G_{n,p}$ are not incident.
            This yields a covariance of
            \begin{align*}
                &
                    p^4(1-p)^2
                    \Bigl(
                        (1-p^2)^{\ell_1-3}(1-p^2)^{\ell_2-3}
                        \bigl(
                            (1-p)^4
                            + 4p(1-p)^3
                            + 2p^2(1-p)^2
                        \bigr)
                        - (1-p^2)^{\ell_1+\ell_2-2}
                    \Bigr)
                \\
                ={}& p^4 (1-p^2)^{l_1-l_2-6} (1-p)^4
                    \bigl(
                        8 + 3p(1-p) + p^3 + p^4
                    \bigr) >0.
            \end{align*} 
            
        \item[$\vert e \cup f \cup g \cup h \vert = 5 $]
            In this case each configuration uses $5$ nodes and therefore at least $3$ arcs, which leads to a maximal contribution of $n^5p^3$, which is dominated by our target rate $n^6p^3$ in all regimes.
            
        \item[$\vert e \cup f \cup g \cup h \vert = 4 $]
            Each configuration here uses $4$ nodes and at least $3$ arcs, as shown \nameCref{fig:variance}~\hyperref[fig:variance]{2h} and can therefore contribute at most on the order of $n^4p^3$, which is dominated by $n^6p^3$ in all regimes.
            
        \item[$\vert e \cup f \cup g \cup h \vert = 3 $]
            Each configuration here uses $3$ nodes and at least $2$ arcs. Therefore the contribution is at most on the order of $n^3p^2$, which is dominated by $n^6p^3$ if $n\gg \frac{1}{n^2}$ and by $n^2p$ if $n\lesssim \frac{1}{n^2}$.
    
        \end{subcases}
        
    \item[One pair corresponds to the origin]
        If one vertex corresponds to the origin the indicator of its pair becomes the indicator of whether the other arc is there. To keep the above notation we abuse notation and treat the origin as an empty arc, i.e., $i=\emptyset$ if vertex $i$ corresponds to the origin.
    
        \begin{subcases}
    
        \item[$|e\cup f \cup g\cup h |= 6 $]
            This case consist of three non-overlapping arcs, one of which, say $e$, is contained in a pair with the origin. As the unique $4$-cycle of the other arcs, does not contain $e$, the indicator variables in this case are independent.
            
        \item[$|e\cup f \cup g\cup h| = 5 $]
            In this case at least $3$ arcs need to be present, which leads to a contribution on the order of $n^5p^3$, which is dominated by our target rate $n^6p^3$ in all regimes.
        
        \item[$|e\cup f \cup g\cup h| \in {3,4} $]
            In this case at least $2$ arcs need to be present, which leads to an overall contribution of at most $n^4p^2$, is dominated by $n^6p^3$ if $p\gg \frac{1}{n^2}$ and does not exceed $n^2p$ if $p\lesssim \frac{1}{n^2}$.
            
        \item[$|e\cup f\cup  g\cup h| = 2 $]
            In this case each pair contains the same arc $e$, possibly with different orientations, and the origin.
            Each pair corresponds to an edge if and only if the other one does.
            So in this case we obtain a contribution of 
            $$
                4\binom{n}{2} (p-p^2) = 4\binom{n}{2}p(1-p),
            $$
            which is dominant in the case $p\lesssim \frac{1}{n^2}$.
        
        \end{subcases}
    \end{maincases}     
    \textbf{Conclusion.}
        As argued in each case separately, the term $n^6p^3$ from \Cref{subcase:SEPTriagvariance:12} is dominant in the regime $p\gg \frac{1}{n^2}$, while in the case $p\lesssim \frac{1}{n^2}$ no term exceeds $n^2p$.
        \qedhere

\end{proof}

\begin{rem}
\begin{itemize}
    \item[(a)] Comparing the variance of the number of edges of the triangulations to that of the symmetric edge polytope derived in \Cref{variance_sep}, we observe that for $p \geq \frac{1}{n^2}$, the asymptotic rate remains unchanged, except that the factor $(1-\sqrt{2}p)^2$ disappeared. For $p \lesssim \frac{1}{n^2}$, we obtain a new lower bound. In this regime of small $p$, the guaranteed edges connected to the origin dominate the interactions between distinct arcs.

    \item[(b)] \Cref{variance_sepTriag} serves as a weaker analogue to \Cref{variance_sep} because it is restricted to the regime $\limsup p < 1$. Consequently, this bound cannot be used to prove a central limit theorem when $p$ converges to $1$. However, we conjecture that the result holds with asymptotic equality and that bounds on $(1-p)$ similar to those for the symmetric edge polytope could be established.
\end{itemize}
\end{rem}

\subsection{Discrete gradients}

Having established a variance lower bound, we now turn to the discrete gradients. As with the variance, the proof strategy for $\SEP_{n,p}$ carries over directly.

\begin{lemma}
    Let $e \in \binom{[n]}{2}$. Then
    \begin{align*}
        \EE\bigl[(\D_{e} K_{n,p})^4\bigr]
        &\lesssim  \begin{cases}
            n^8 p^6q^2, & \text{if } p \gg n^{-2}, \\
            n^2 p^3q^2, & \text{if } p \lesssim n^{-2}.
        \end{cases}
    \end{align*}
\end{lemma} 
\begin{proof}  
    The proof of \Cref{Dfbound_SEP} goes through almost verbatim. For every arc in $G$, at most $4$ edges are added; here, there are an additional two edges connecting to the origin. We remove at most $2W_2$ edges due to $3$-cycles, and for every $4$-cycle, we might remove up to $6W_3$ edges. This yields the intermediate bound
    $
        \D_e \le 4 E_{n,p} + 2 + 2 W_2 + 6 W_3.    
    $
    The result then follows by applying \Cref{l:W123_bound_sep}. 
\end{proof}

For the second-order discrete gradient, we again rely on the previous proof but slightly refine the reasoning.

\begin{lemma}
     Let $K_{n,p}$ be as before. Let $e,f\in \binom{[n]}{2}$ with $e\neq f$. Then,
    for $e\cap f = \emptyset$ it holds
    \begin{align*}
        \E\left[(\D_e\D_fK_{n,p})^4\right]
        &\lesssim p^4q^4,
    \intertext{and for $\vert e\cap f\vert= 1$ we have the bounds}
        \E\left[(\D_e\D_fK_{n,p})^4\right]
        &\lesssim
        p^4q^4 + n^4p^{12}q^4
        \simeq
        \begin{cases}
            n^4p^{12}q^4, &\text{if } p \gg n^{-\frac12},\\
            p^4q^4, &\text{if } p \lesssim n^{-\frac12}.
        \end{cases}
    \end{align*}
\end{lemma}
\begin{proof}
    This proof utilizes the same arguments as \Cref{DDbounds_SEP} and adopts the variable names depicted in \Cref{fig:DDvarNames}.
    
    We consider the same cases, i.e., arc configurations within $G_{n,p}$, and apply analogous arguments to achieve the same result. We assume $e$ and $f$ are the smallest arcs, meaning they eliminate all edges for pairs with which they form a $4$-cycle, just as with the polytope.  
    
    In the proof of \Cref{DDbounds_SEP}, we used two primary arguments to bound each case. The first approach trivially bounded the case almost surely by counting the possible configurations of the arcs. This logic, which was used to bound \hyperref[bound:Sef]{$S_{e,f}$}, \hyperref[bound:T-f-e]{$T_{e}^{(f)}, T_{f}^{(e)}$}, and the first case of \hyperref[bound:Tef]{$T_{e,f}$}, carries over verbatim.
    
    The second approach demonstrated that any two pairs possessing a certain property formed a $4$-cycle, implying that only one such pair could form an edge. This argument was used to bound \hyperref[bound:Ref]{$R_{e,f}$} and the second case of \hyperref[bound:Tef]{$T_{e,f}$}. In the polytope proof, the existence of such a $4$-cycle immediately contradicted either pair forming an edge. For triangulations, while this immediate contradiction no longer holds, the ultimate conclusion remains valid. 
    
    Specifically, if every two pairs in a set $S$ of arc pairs form a $4$-cycle with each other, only one pair in $S$, say $e_1,e_2$, can be the smallest. 
    Since every other pair in $S$ forms a $4$-cycle with $e_1,e_2$, only one pair in $S$ can actually form an edge. Applying this reasoning to our cases yields the same bounds as those for the symmetric edge polytope.

    The only remaining case involves an edge between the origin and a vertex corresponding to an arc $g$. 
    Because all edges connected to the origin are independent of all other arcs except $g$ itself, this contributes $0$ to $\D_e \D_f K_{n,p}$.

    Thus, the arguments above yield the same overall bounds as in \Cref{DDbounds_SEP}.    
\end{proof}

\subsection{Central limit theorem}

Having computed the same bounds for $K_{n,p}$ as for the symmetric edge polytope in the regime $\limsup p < 1$, we can now state a similar central limit theorem.

\begin{thm}\label{thm:CLTtriangulations}
    Let $G \sim G_{n,p}$, let $K_{n,p}$ be the number of edges of any unimodular triangulation of $\SEP_G$, and let $N \sim N(0,1)$. If $p \gg \frac{1}{n^2}$ and $\limsup_{n \to \infty} p < 1$, then we have
    $$
        d_K\left(\frac{K_{n,p}-\E[K_{n,p}]}{\sqrt{\V(K_{n,p})}}, N\right)
        \lesssim 
        \frac{1}{n\sqrt{p}}.
    $$
    In particular, $K_{n,p}$ satisfies a central limit theorem as $n \to \infty$.
\end{thm}
\begin{proof}
    Since we have established the same bounds for both $\D_e K_{n,p}$ and $\D_e \D_f K_{n,p}$ as in the symmetric edge polytope case, we obtain identical bounds for the terms $B_1,\ldots,B_5$. This leads to
    $$
    \begin{aligned}
        d_K\left(\frac{K_{n,p}-\E[K_{n,p}]}{\sqrt{\V(K_{n,p})}}, N\right) 
        &\simeq \frac{n^{5}p^{5/2}\sqrt{q}}{\V(K_{n,p})}.
    \end{aligned}
    $$
    Applying the lower bound $\V(K_{n,p}) \gtrsim n^6p^3$ for $\limsup_{n \to \infty} p < 1$ yields
    $$
    \begin{aligned}
        d_K\left(\frac{K_{n,p}-\E[K_{n,p}]}{\sqrt{\V(K_{n,p})}}, N\right) 
        &\lesssim \frac{n^{5}p^{5/2}}{n^6p^3} \simeq \frac{1}{n\sqrt{p}},
    \end{aligned}
    $$
    which completes the proof.
\end{proof}

\subsubsection*{Acknowledgement}
This work has been supported by the German Research Foundation (DFG) via SPP 2458 \textit{Combinatorial Synergies} and via project number 547295909.

\bibliographystyle{plain}  
\bibliography{quellen}   

\end{document}